\documentclass{article}
\usepackage[a4paper,margin=1in]{geometry}
\usepackage[utf8]{inputenc}
\usepackage{lmodern}
\usepackage[leqno]{amsmath}
\usepackage{amssymb,amsfonts,mathtools}\allowdisplaybreaks
\usepackage{bm}
\usepackage{microtype}
\usepackage{algorithm}
\usepackage{graphicx}
\graphicspath{{figure/}}
\usepackage[caption=false]{subfig}
\usepackage{epstopdf}
\usepackage[amsmath,thmmarks,hyperref]{ntheorem}
\usepackage[hidelinks]{hyperref}
\usepackage[capitalize,nameinlink]{cleveref}
\usepackage{algpseudocode}
\crefname{section}{section}{sections}
\crefname{subsection}{subsection}{subsections}
\crefname{subsubsection}{subsubsection}{subsubsections}

\newcommand{\vect}[1]{\bm{#1}}

\newcommand{\tens}[1]{\bm{\mathsf{#1}}}

\newcommand{\ii}{\mathrm{i}}
\newcommand{\ee}{\mathrm{e}}
\newcommand{\C}{\mathbb{C}}
\newcommand{\R}{\mathbb{R}}

\newcommand{\eps}{\varepsilon}
\newcommand{\calA}{\mathcal{A}}

\newcommand{\calC}{\mathcal{C}}

\newcommand{\calH}{\mathcal{H}}
\newcommand{\HankelSpace}{\mathfrak{H}}
\newcommand{\calG}{\mathcal{G}}

\newcommand{\calP}{\mathcal{P}}
\newcommand{\calQ}{\mathcal{Q}}
\newcommand{\calR}{\mathcal{R}}

\newcommand{\calT}{\mathcal{T}}
\newcommand{\calU}{\mathcal{U}}

\newcommand{\calK}{\mathcal{K}}
\newcommand{\calI}{\mathcal{I}}
\newcommand{\rank}{\operatorname{rank}}

\newcommand{\range}{\operatorname{range}}

\newcommand{\Arg}{\operatorname{Arg}}
\newcommand{\tr}{\operatorname{tr}}

\newcommand{\one}{\mathbf{1}}

\newcommand{\norm}[1]{\left\lVert #1\right\rVert}
\newcommand{\abs}[1]{\left\lvert #1\right\rvert}

\newcommand{\inner}[2]{\left\langle #1,#2\right\rangle_F}
\makeatletter
\renewtheoremstyle{plain}
 {\item[\hskip\labelsep\hskip\parindent \theorem@headerfont ##1\ \textup{##2\theorem@separator}]}%
 {\item[\hskip\labelsep\hskip\parindent \theorem@headerfont ##1\ \textup{##2}]\textup{(##3)\theorem@separator}\ }

\renewtheoremstyle{nonumberplain}%
 {\item[\theorem@headerfont\hskip\labelsep\hskip\parindent ##1\theorem@separator]}%
{\item[\theorem@headerfont\hskip\labelsep\hskip\parindent ##3\theorem@separator]}

\newcommand{\proofbox}{\vbox{\hrule height0.6pt\hbox{\vrule height1.3ex width0.6pt\hskip0.8ex\vrule width0.6pt}\hrule height0.6pt}}

\theorempreskip{1ex plus .25ex minus .1ex}
\theorempostskip{1ex plus .25ex minus .1ex}

\theoremstyle{nonumberplain}
\theoremheaderfont{\normalfont\itshape}
\theorembodyfont{\normalfont}
\theoremseparator{.}
\theoremsymbol{\proofbox}
\newtheorem{proof}{Proof}

\theoremstyle{plain}
\theoremheaderfont{\normalfont\sc}
\theorembodyfont{\normalfont\itshape}
\theoremseparator{.}
\theoremsymbol{}
\newtheorem{theorem}{Theorem}

\newcommand{\newsiamthm}[2]{
  \theoremstyle{plain}
  \theoremheaderfont{\normalfont\sc}
  \theorembodyfont{\normalfont\itshape}
  \theoremseparator{.}
  \theoremsymbol{}
  \newtheorem{#1}[theorem]{#2}
}

\newsiamthm{lemma}{Lemma}
\newsiamthm{corollary}{Corollary}
\newsiamthm{proposition}{Proposition}
\newsiamthm{definition}{Definition}

\newcommand{\newsiamremark}[2]{
  \theoremstyle{plain}
  \theoremheaderfont{\normalfont\itshape}
  \theorembodyfont{\normalfont}
  \theoremseparator{.}
  \theoremsymbol{}
  \newtheorem{#1}[theorem]{#2}
}
\makeatother
\numberwithin{theorem}{section}
\numberwithin{equation}{section}
\numberwithin{algorithm}{section}
\newsiamremark{remark}{Remark}
\newsiamremark{assumption}{Assumption}
\newsiamremark{condition}{Condition}
\newsiamremark{example}{Example}

\date{}
\title{Off-Grid Point-Scatterer Localization from Sparse Limited-Aperture Data via Hankel Completion}

\author{
Jingzhi Li\thanks{Department of Mathematics, Southern University of Science and Technology, \& Shenzhen International Center for Mathematics, Shenzhen, China. email: li.jz@sustech.edu.cn}
\and
Xiliang Lu\thanks{School of Mathematics and Statistics, Hubei Center for Applied Mathematics, and Hubei Key Laboratory of Computational Science, Wuhan University, Wuhan 430072, China. email: xllv.math@whu.edu.cn}
\and
Juntao You\thanks{(Corresponding author.) School of Artificial Intelligence, Hubei Center for Applied Mathematics, and Hubei Key Laboratory of Computational Science, Wuhan University, Wuhan 430072, China. email: youjuntao@whu.edu.cn }
}

\newenvironment{keywords}{\par\smallskip\noindent\textbf{Keywords.} }{\par}
\newenvironment{MSCcodes}{\par\smallskip\noindent\textbf{MSC codes.} }{\par}

\begin{document}
\maketitle

\begin{abstract}
We study the localization of $s$ distinct planar point scatterers from sparsely sampled single-incidence, single-frequency acoustic far-field data over a limited receiver aperture. Under the first Born approximation, the coherent data form a two-dimensional off-grid exponential sum and admit a low-rank Hankel representation. A limited aperture generally admits multiple aperture-admissible Hankel pencils with different lifted dimensions, sampling multiplicities, and Vandermonde conditioning. Hence, the choice of the Hankel pencil affects both data completion and subsequent localization.

In this work, we introduce a unified framework for structured completion and off-grid localization based on general aperture admissible Hankel pencils. We establish exact recovery in the noiseless case and stable recovery under bounded noise from a number of random aperture samples that is linear in $s$, up to factors determined by the pencil geometry, source conditioning, and logarithmic terms, together with an explicit localization bound in terms of the completion error. The analysis separates the contribution of the pencil geometry from the conditioning induced by the source configuration and provides verifiable sufficient conditions for stable completion and MUSIC resolution. These results yield quantitative criteria for selecting sample efficient Hankel pencils under a prescribed physical aperture. Numerical experiments on connected and disconnected apertures demonstrate the effects of pencil geometry, sampling, and noise.
\end{abstract}

\begin{keywords}
 Inverse acoustic scattering, point scatterers, limited aperture, off-grid localization, spectral compressed sensing, low-rank Hankel completion, MUSIC
\end{keywords}

\begin{MSCcodes}
35R30, 15A83, 65F55, 65N21, 94A12
\end{MSCcodes}

\section{Introduction}\label{sec:intro}

Inverse scattering aims to determine unknown scatterers or material inhomogeneities from measurements of scattered waves. It is fundamental in radar and sonar imaging, remote sensing, geophysical exploration, medical imaging, and nondestructive evaluation \cite{ColtonKress2019,CakoniColtonHaddar2022}. In many applications, the unknown scene is sparse or can be well approximated by a finite collection of localized scatterers \cite{AmmariIakovlevaMoskow2003,Fannjiang2010I,Fannjiang2010II}. A central challenge is therefore to recover their off-grid locations from sparse measurements available only over a restricted observation geometry.

Data-limited inverse acoustic scattering has motivated sampling methods for one or a few incident fields and for restricted or nonstandard acquisition configurations \cite{Fannjiang2011,ItoJinZou2012,LiZou2013,KangLim2022,NingZou2025,GarnierHaddarMontanelli2023}, as well as data-recovery approaches that infer effective full-aperture information from limited-aperture measurements \cite{LiuSun2019,DouLiuMengZhang2022}. We focus on single-incidence, fixed-frequency acoustic scattering from finitely many point scatterers supported on a known plane \cite{ChallaSini2012}; see \Cref{fig:ainverseofs}. Measurements are available only over a limited receiver aperture, and only a sparse subset of the accessible data is observed. The aperture determines which scattering information is physically accessible, whereas the sampling pattern determines which of these data is actually acquired. In practice, the accessible aperture need not have a simple connected geometry; occlusion by obstacles, restrictions on sensor placement or platform motion, limited fields of view, and other acquisition constraints may lead to nonrectangular or even disconnected observation regions.

Under the first Born approximation \cite{Natterer2004,AmmariIakovlevaMoskow2003}, the far-field measurements can be represented as samples of a two-dimensional off-grid exponential sum, with the scatterer locations encoded continuously in its phases. The receiver aperture therefore induces a geometry-constrained frequency domain, while sparse acquisition provides only incomplete observations over that domain. The inverse scattering problem is thus reduced to recovering and localizing a sparse continuous exponential model from partial data on a frequency set whose geometry is inherited directly from the physical acquisition configuration.

The resulting exponential structure connects the problem to off-grid sparse recovery and spectral super-resolution \cite{TangBhaskarShahRecht2013,CandesFernandezGranda2014,CatalaDuvalPeyre2019}, with continuous sparse formulations also arising directly in point-scatterer inverse scattering \cite{AlbertiPetitSantacesaria2024}. General-domain Hankel structures and frequency estimation on nonstandard index sets have been studied in \cite{AnderssonCarlsson2015,AnderssonCarlsson2018}. The sparse-observation Hankel recovery guarantees in \cite{ChenChi2014,CaiWangWei2019,CaiCaiYou2023,CaiHuangLuYou2025}, however, are formulated for prescribed interval or rectangular pencils. In limited-aperture scattering, the accessible frequency domain can be nonrectangular or disconnected and may admit multiple aperture-admissible pencils with different lifted dimensions, sampling multiplicities, and Vandermonde conditioning. We therefore establish exact and stable sparse-sampling recovery for general aperture-admissible Hankel pencils, quantify the dependence on the pencil geometry and source conditioning, and propagate the completion error to subsequent off-grid localization.

Our main contributions are as follows.
\begin{enumerate}
\item[(i)] We develop a unified reconstruction framework based on aperture-admissible Hankel pencils $A,B\subset\mathbb Z^2$ satisfying $A+B\subset S_{\rm ap}$. The missing coherent data are recovered by structured low-rank completion, followed by two-sided MUSIC for localization and least squares for amplitude estimation.

\item[(ii)] We establish exact recovery with high probability from $O(\Gamma_{A,B}\mu_1c_{A,B}s\log^2|S|)$ samples, together with stability under bounded noise and an explicit localization bound. Under verifiable geometric and minimum separation conditions, this reduces to $O(s\log^4|S|)$ samples. The theory quantitatively relates the sampling complexity and localization stability to the pencil geometry and the number of scatterers.

\item[(iii)]  We derive explicit bounds on the geometric and conditioning factors: rectangular decompositions control $\Gamma_{A,B}$, while Cartesian cores and minimum separation control the Vandermonde conditioning and MUSIC resolution. These bounds yield quantitative criteria for selecting sample efficient aperture admissible pencils.

\item[(iv)] Numerical experiments on connected and disconnected apertures validate the proposed method and the theoretical criteria for pencil selection under sparse sampling and noise.
\end{enumerate}

\begin{figure}[t]
\centering
{%
\includegraphics[width=0.68\textwidth,
    trim=0.5cm 0.5cm 0cm 0.8cm,clip]{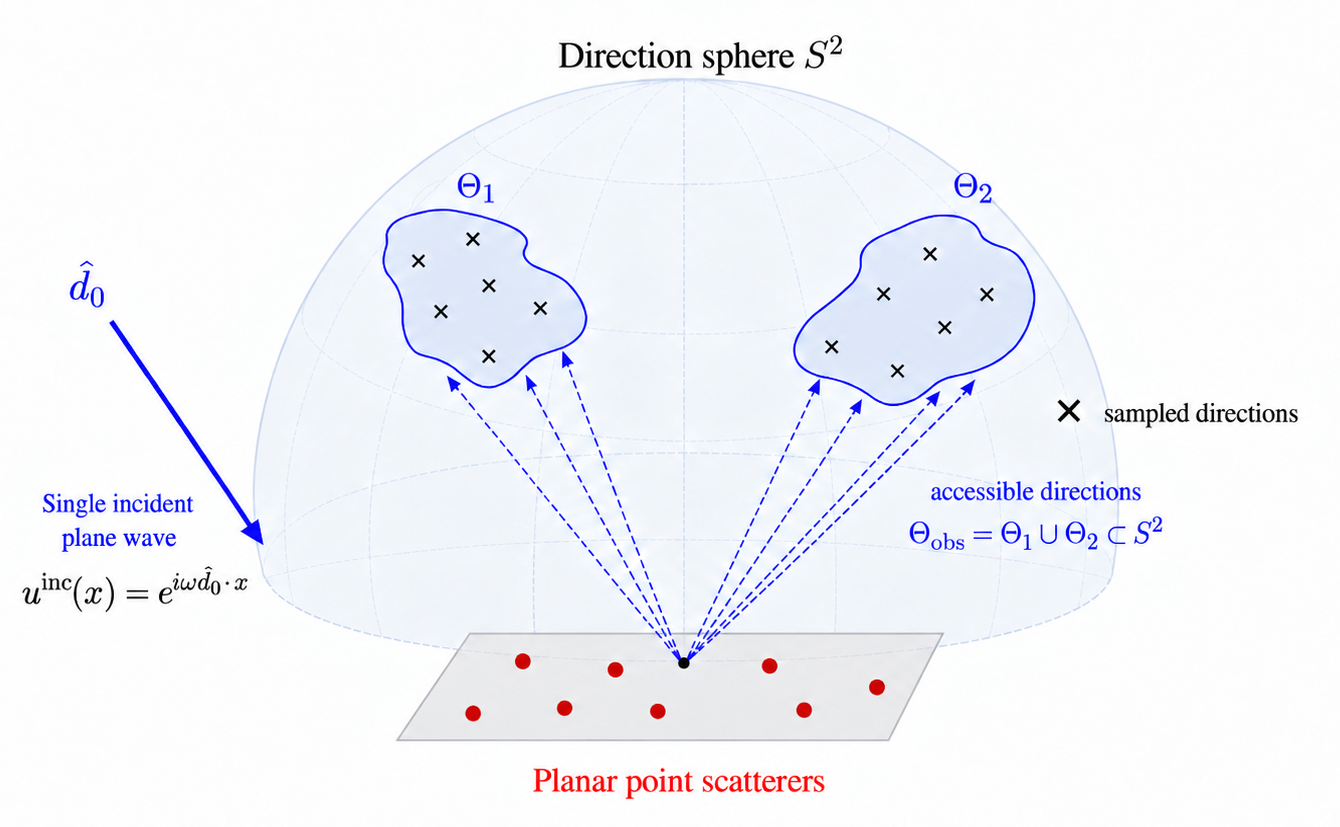}%
}{}
\caption{Single-incidence, single-frequency scattering from planar point targets, with sparse observations collected over a limited far-field acquisition surface.}
\label{fig:ainverseofs}
\end{figure}
\paragraph{Notation and organization}
Bold lowercase letters denote vectors and bold sans-serif uppercase letters denote matrices. For a scalar $a$, $\abs{a}$ denotes its modulus, while $|S|$ and $\#S$ denote the cardinality of a finite set $S$. For a vector $\vect x$, $\norm{\vect x}_2$ and $\norm{\vect x}_\infty$ denote the Euclidean and maximum norms, respectively. For a matrix $\tens M$, $\norm{\tens M}_2$, $\norm{\tens M}_F$, and $\norm{\tens M}_*$ denote the spectral, Frobenius, and nuclear norms, respectively, and $\tens M^T$, $\tens M^*$, and $\overline{\tens M}$ denote its transpose, conjugate transpose, and entrywise complex conjugate. For a linear operator $\mathcal L$ on a matrix space, $\norm{\mathcal L} :=\sup_{\norm{\tens M}_F=1}\norm{\mathcal L(\tens M)}_F$. We write $\inner{\tens X}{\tens Y}=\tr(\tens X^*\tens Y)$ for the Frobenius inner product, $\sigma_j(\tens M)$ for the $j$th largest singular value of $\tens M$, and $\lambda_{\min}(\tens M)$ for the smallest eigenvalue of a Hermitian matrix $\tens M$.  We write $f=\mathcal O(g)$ if $\abs{f}\le Cg$ for some constant $C>0$ independent of the asymptotic variables specified in the corresponding statement.

Section~\ref{sec:model} specifies the Born data model, Section~\ref{sec:hankel} gives the reconstruction, and Section~\ref{sec:conditioning} states the recovery and localization results. Sections~\ref{sec:experiments} and~\ref{sec:conclusion} present the experiments and concluding remarks. The proofs are collected in the appendix. 

\section{Born data model and aperture discretization}\label{sec:model}

\subsection{Planar point-scatterer model}

Let $u^{\rm inc}(\vect x;\hat{\vect d})= \exp(\ii\omega\hat{\vect d}\cdot\vect x)$, where $\omega>0$. For a compactly supported contrast $\eta$, the total field solves
\[
\Delta u^t+\omega^2(1+\eta)u^t=0\quad\text{in }\mathbb R^3,
\qquad u^t=u^{\rm inc}+u^s,
\]
with $u^s$ outgoing. Replacing $u^t$ by $u^{\rm inc}$ in the scattering integral gives the first Born far-field pattern
\begin{equation}\label{eq:born-ff-new}
 u_{\rm B}^\infty(\hat{\vect r},\hat{\vect d})
 =\frac{\omega^2}{4\pi}\int_{\mathbb R^3}\eta(\vect y)
 \exp\!\bigl[-\ii\omega(\hat{\vect r}-\hat{\vect d})\cdot\vect y\bigr]
 \,\mathrm d\vect y.
\end{equation}
The Fourier integral extends to compactly supported finite complex Radon measures. We take
\[
 \eta=\sum_{j=1}^s\xi_j\delta_{(t_{1,j},t_{2,j},0)^T},
 \qquad \xi_j\in\mathbb C\setminus\{0\},
\]
fix the incident direction $\hat{\vect d}_0$, and set
\[
 \vect t_j=(t_{1,j},t_{2,j})^T,
 \qquad
 c_j=\frac{\omega^2}{4\pi}\xi_j
 \exp\!\bigl[\ii\omega(t_{1,j},t_{2,j},0)^T\cdot\hat{\vect d}_0\bigr].
\]
With $\vect q=\omega(\hat r_1,\hat r_2)^T$, equation~\eqref{eq:born-ff-new} becomes
\begin{equation}\label{eq:Aq-new}
 F(\vect q)=\sum_{j=1}^s c_j\exp(-\ii\vect q\cdot\vect t_j).
\end{equation}
Equation~\eqref{eq:Aq-new} is the forward model analyzed below. Measurement noise, finite-radius effects, multiple scattering, and other departures from the Born point-scatterer model are included in the perturbation term introduced in \eqref{eq:partial-aperture-observations}.

\begin{remark}[Two-reference intensity data]
If only intensities are measured, let $r_m(\vect q)=\alpha_m\exp(-\ii\vect q\cdot\vect a_m)$, $m=1,2$, be calibrated reference fields. From $I_0=|F|^2$ and $I_m=|F+r_m|^2$, one obtains $[I_m-I_0-|r_m|^2]/2=\operatorname{Re}(\overline r_m F)$. The resulting two real equations determine the complex value $F(\vect q)$ whenever $\operatorname{Im}(\overline r_1r_2)\ne0$. The subsequent analysis starts from the coherent samples obtained directly or by this pointwise preprocessing.
\end{remark}

\subsection{Limited aperture and the recovery problem}
\label{subsec:aperture-disk}

Let the receiver locations be $\rho(\hat{\vect r})\hat{\vect r}$ for $\hat{\vect r}\in\Theta_{\rm obs}\subset\{\hat{\vect r}\in\mathbb S^2: \hat r_3>0\}$, with $\rho(\hat{\vect r})\ge R_0$. The demodulated finite-radius field is
\begin{equation}\label{eq:finite-radius-demodulated}
 F_{\rm fr}(\hat{\vect r})
 :=\rho(\hat{\vect r})\exp[-\ii\omega\rho(\hat{\vect r})]
 u^s(\rho(\hat{\vect r})\hat{\vect r};\hat{\vect d}_0).
\end{equation}
The standard far-field expansion gives $F_{\rm fr}(\hat{\vect r})=u^\infty(\hat{\vect r},\hat{\vect d}_0) +O(R_0^{-1})$ uniformly on the accessible directions; under the Born model, $u^\infty$ is replaced by \eqref{eq:Aq-new}.

The accessible transverse-frequency region is
\[
 D_q=\{\omega(\hat r_1,\hat r_2)^T:
 \hat{\vect r}\in\Theta_{\rm obs}\}\subset\{\vect q:\|\vect q\|_2\le\omega\}.
\]
For example, a conical aperture of half-angle $\theta_{\max}$ gives the disk $\|\vect q\|_2\le\omega\sin\theta_{\max}$. Choose spacings $\Delta q_1,\Delta q_2>0$ and define
\begin{equation}\label{eq:S-ap-def}
\begin{aligned}
 \vect q_k&=(q_{1,0}+k_1\Delta q_1,
              q_{2,0}+k_2\Delta q_2)^T,\\
 S_{\rm ap}&=\{k\in\mathbb Z^2:\vect q_k\in D_q\}.
\end{aligned}
\end{equation}
The set $S_{\rm ap}$ is finite because $D_q$ is bounded. We observe the far-field data only on a subset $\Omega_{\rm obs}\subset S_{\rm ap}$:
\begin{equation}\label{eq:partial-aperture-observations}
 y_k=F(\vect q_k)+e_k,\qquad k\in\Omega_{\rm obs}.
\end{equation}
Then, the objective is to recover the locations and strengths $\{(\vect t_j,\xi_j)\}_{j=1}^s$ from these partial data $\{y_k\}_{k\in\Omega_{\rm obs}}$.

\section{Partial-aperture data recovery via low-rank Hankel completion}\label{sec:hankel}

The exponential-sum model above leads to a two-stage reconstruction: we first complete the coefficients on an aperture-admissible sum set $S=A+B\subset S_{\rm ap}$ and then estimate the continuous scatterer parameters from the completed Hankel matrix. This section formulates both stages, while the recovery and localization guarantees are developed in \Cref{sec:conditioning}.

\subsection{Aperture-admissible Hankel pencils}\label{subsec:hankel-structure}

To exploit the finite point-scatterer structure in the partial aperture data, we first express the noiseless samples on the Cartesian frequency grid. Write $x_k=F(\vect q_k)$. Substituting \eqref{eq:S-ap-def} into \eqref{eq:Aq-new} yields
\begin{equation}\label{eq:X-expsum-new}
x_{k_1,k_2}=\sum_{j=1}^s\widetilde c_j z_{1,j}^{k_1}z_{2,j}^{k_2},
\quad
\widetilde c_j=c_j\exp[-\ii(q_{1,0}t_{1,j}+q_{2,0}t_{2,j})],
\quad
z_{\ell,j}=\exp(-\ii\Delta q_\ell t_{\ell,j}).
\end{equation}
Thus the aperture data form a finite two-dimensional exponential sum on $S_{\rm ap}$.

To incorporate the aperture geometry into the Hankel lifting, let $A,B\subset\mathbb Z^2$ be nonempty finite sets and set $S=A+B$. For $\vect x=(x_k)_{k\in S}$, define
\[
[\calH_{A,B}(\vect x)]_{\alpha,\beta}=x_{\alpha+\beta},
\qquad
\alpha\in A,\quad \beta\in B.
\]
We call $(A,B)$ \emph{aperture-admissible} if $A+B\subset S_{\rm ap}$. This condition ensures that every entry of the lifted matrix is indexed by a frequency accessible through the prescribed aperture. Rectangular choices recover the usual block-Hankel lifting, whereas general-domain Hankel structures allow nonstandard index sets \cite{AnderssonCarlsson2015,AnderssonCarlsson2018}. Here admissibility additionally requires $A+B\subset S_{\rm ap}$. For example, \Cref{fig:aperture-admissible-pencil} illustrates this construction for a disconnected aperture; nonrectangular or disconnected accessible frequency sets (see also \Cref{fig:sector-music}) similarly lead to nonstandard admissible pencils.

\begin{figure}[t]
\centering
\includegraphics[width=0.82\linewidth]{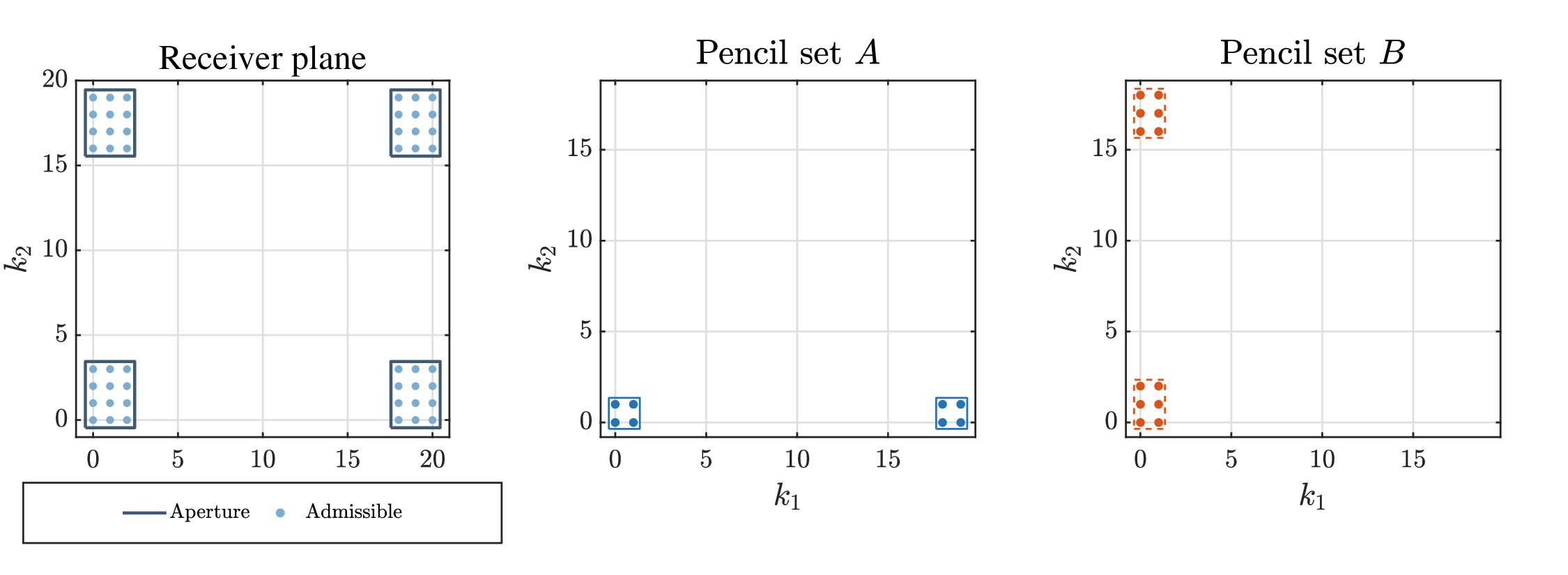}
\caption{Illustration of an aperture-admissible Hankel pencil for a disconnected acquisition geometry. The receiver aperture induces a nonrectangular accessible frequency set $S_{\rm ap}$. Finite index sets $A$ and $B$ satisfying $A+B\subset S_{\rm ap}$ define a Hankel pencil compatible with the physical aperture.}
\label{fig:aperture-admissible-pencil}
\end{figure}

For any aperture-admissible pair $(A,B)$, the additive indexing of the exponential data in \eqref{eq:X-expsum-new} gives
\begin{equation}\label{eq:pencil-factor-new}
\tens M_\star=\calH_{A,B}(\vect x)=\tens V_A\tens D\tens V_B^T,
\qquad
[\tens V_A]_{\alpha,j}=z_{1,j}^{\alpha_1}z_{2,j}^{\alpha_2},
\quad
[\tens V_B]_{\beta,j}=z_{1,j}^{\beta_1}z_{2,j}^{\beta_2},
\end{equation}
where $\tens D=\operatorname{diag}(\widetilde c_1,\ldots,\widetilde c_s)$. The factorization yields the following rank property.

\begin{proposition}\label{prop:pencil-rank-physical}
For any aperture-admissible pair $(A,B)$, we have
\[
\rank\calH_{A,B}(\vect x)\le s.
\]
If $\widetilde c_j\neq0$ for all $j$ and both $\tens V_A$ and $\tens V_B$ have full column rank, then $\rank\calH_{A,B}(\vect x)=s$.
\end{proposition}

The low-rank structure above turns the recovery of unobserved aperture samples into a structured matrix completion problem. Sparse Hankel completion has been studied extensively for rectangular pencils through both convex and nonconvex approaches \cite{ChenChi2014,CaiWangWei2019,CaiCaiYou2023,CaiHuangLuYou2025}. Here, however, the pencil is constrained by $A+B\subset S_{\rm ap}$ and hence by the aperture geometry. Different admissible choices may have different dimensions, fiber multiplicities, and source-dependent Vandermonde conditioning, leading to different recovery guarantees. We therefore develop a completion theory for general aperture-admissible pencils that retains the dependence on both the pencil geometry and the source configuration; see \Cref{sec:completion-localization-guarantees}.

\subsection{Completion and parameter estimation}\label{subsec:completion-models}

The factorization in \Cref{subsec:hankel-structure} reduces partial-aperture data recovery to structured matrix completion. For rectangular Cartesian domains, related convex and nonconvex Hankel recovery methods are well developed \cite{ChenChi2014,CaiWangWei2019,CaiCaiYou2023,CaiHuangLuYou2025}. In the present setting the admissible pencil itself is selected from the factorizations $A+B\subset S_{\rm ap}$, so its lifted dimensions, fiber multiplicities, and Vandermonde conditioning enter the recovery guarantee. We therefore formulate the completion problem directly on the selected finite pencil.

For a selected aperture-admissible pencil, set $\Omega=\Omega_{\rm obs}\cap S$. With noiseless observations, we solve
\begin{equation}\label{eq:admissible-emac-program-new}
 \min_{\vect u\in\mathbb C^S}\|\calH_{A,B}(\vect u)\|_*
 \quad\text{subject to}\quad u_k=y_k,
 \qquad k\in\Omega.
\end{equation}
If $y_k=x_k+e_k$ and $\|\calP_{\Omega}\vect e\|_2\le\eps$, the noisy estimator is
\begin{equation}\label{eq:physical-noisy-emac-program-new}
 \min_{\vect u\in\mathbb C^S}\|\calH_{A,B}(\vect u)\|_*
 \quad\text{subject to}\quad
 \|\calP_{\Omega}(\vect u-\vect y)\|_2\le\eps.
\end{equation}
These programs are convex. Algorithm~\ref{alg:admm-hankel-completion} records the scaled ADMM implementation used in the experiments. Define the fiber multiplicity $w(k)=\#\{(\alpha,\beta)\in A\times B:\alpha+\beta=k\}$ and
\[
 [\calH_{A,B}^{\dagger}(\tens M)]_k
 =\frac{1}{w(k)}\sum_{\alpha+\beta=k}M_{\alpha,\beta}.
\]
The coefficient update is the weighted projection of $\calH_{A,B}^{\dagger}(\tens Z^{t+1}+\tens Y^t)$ onto the feasible set, because $\|\calH_{A,B}(\vect u)-\tens M\|_F^2 =\sum_{k\in S}w(k)|u_k-[\calH_{A,B}^{\dagger}(\tens M)]_k|^2+ \text{const}$.

\begin{algorithm}[t]
\small
\caption{ADMM for aperture-admissible Hankel completion}
\label{alg:admm-hankel-completion}
\begin{algorithmic}[1]
\State \textbf{Input:} $\{y_k:k\in\Omega\}$, $A,B$, penalty
$\tau>0$, and $\eps$ for \eqref{eq:physical-noisy-emac-program-new}, maximum iteration $T_{\max}$.
\State Set $x_k^0=y_k$ on $\Omega$ and $x_k^0=0$ otherwise;
$\tens Z^0=\calH_{A,B}(\vect x^0)$, $\tens Y^0=0$, and let $\calC$ be the
feasible set in \eqref{eq:admissible-emac-program-new} or
\eqref{eq:physical-noisy-emac-program-new}.
\For{$t=0,1,2,\ldots, T_{\max}$}
\State $\tens Z^{t+1}=\mathcal D_{1/\tau}
(\calH_{A,B}(\vect x^t)-\tens Y^t)$, where $\mathcal D_\lambda$ applies
singular-value soft thresholding at level $\lambda$.
\State $\displaystyle
 \vect x^{t+1}\in\arg\min_{\vect u\in\calC}
 \|\calH_{A,B}(\vect u)-\tens Z^{t+1}-\tens Y^t\|_F^2$.
\State $\tens Y^{t+1}=\tens Y^t+\tens Z^{t+1}
-\calH_{A,B}(\vect x^{t+1})$.
\State Stop when the primal, dual, and feasibility residuals satisfy the prescribed tolerances.
\EndFor
\State \textbf{Output:} $\widehat{\vect x}$ on $S=A+B$.
\end{algorithmic}
\end{algorithm}

After completion, we estimate the locations from both rank-$s$ singular subspaces using MUSIC \cite{Liao2015} and then recover the amplitudes by least squares. Let $\mathcal F\subset\mathbb R^2$ be a compact rectangular field of view. We assume that the true locations lie in the interior of $\mathcal F$. For $C\in\{A,B\}$, define 
\[
 [\vect v_C(\vect t)]_\alpha
 =\exp[-\ii(\alpha_1\Delta q_1t_1+\alpha_2\Delta q_2t_2)],
 \qquad \vect t=(t_1,t_2)^T.
\]
If $\widehat{\tens M}=\widehat{\tens U}\widehat{\tens\Sigma} \widehat{\tens V}^*$ is the rank-$s$ truncated SVD of $\calH_{A,B}(\widehat{\vect x})$, the two-sided MUSIC residual is
\begin{equation}\label{eq:music-residual}
 \widehat{\calR}(\vect t)=\left[
 \frac{\|(\tens I-\widehat{\tens U}\widehat{\tens U}^*)
 \vect v_A(\vect t)\|_2^2}{|A|}
 +\frac{\|(\tens I-\widehat{\tens V}\widehat{\tens V}^*)
 \overline{\vect v_B(\vect t)}\|_2^2}{|B|}
 \right]^{1/2}.
\end{equation}
The grid induces the periods
\begin{equation}\label{eq:phase-location-new}
 t_1\equiv-\Arg(z_1)/\Delta q_1\pmod{2\pi/\Delta q_1},
 \quad
 t_2\equiv-\Arg(z_2)/\Delta q_2\pmod{2\pi/\Delta q_2}.
\end{equation}
We assume that the side lengths of $\mathcal F$ are smaller than these periods. For $\vect t=(t_1,t_2)^T$ and $\vect t'=(t_1',t_2')^T$, let
\[
 d_{\rm per}(\vect t,\vect t')=
 \left[\min_{m\in\mathbb Z}|t_1-t_1'-2\pi m/\Delta q_1|^2
 +\min_{m\in\mathbb Z}|t_2-t_2'-2\pi m/\Delta q_2|^2\right]^{1/2}.
\]
A source set is $\Delta$-separated if all pairwise periodic distances are at least $\Delta$; when $s=1$, $\Delta>0$ is a prescribed localization scale.

\begin{algorithm}[htbp]
\small
\caption{Two-sided general-domain MUSIC}
\label{alg:general-domain-music}
\begin{algorithmic}[1]
\State \textbf{Input:} $\widehat{\vect x}$, $A,B$, rank $s$, and field of view $\mathcal F$.
\State Form $\widehat{\tens M}=\calH_{A,B}(\widehat{\vect x})$ and compute its
leading $s$ left and right singular vectors.
\State Form \eqref{eq:music-residual} and select its $s$ smallest local minima.
\State \textbf{Output:} $\{\widehat{\vect t}_j\}_{j=1}^s$.
\end{algorithmic}
\end{algorithm}
In computation, the continuous minimization in Algorithm~\ref{alg:general-domain-music} is initialized on a search grid and followed by local refinement. Here and below, local minima are understood relative to $\mathcal F$. The perturbation result below concerns the ideal continuous selection step and does not include search-grid or nonlinear-solver error.

Once the locations are available, the amplitudes are obtained from
\[
 \widehat{\vect c}\in\arg\min_{\vect a\in\mathbb C^s}
 \|\widehat{\vect x}-\tens\Phi_S(\widehat{\vect t})\vect a\|_2,
 \qquad
 [\tens\Phi_S(\widehat{\vect t})]_{k,j}
 =\exp(-\ii\vect q_k\cdot\widehat{\vect t}_j),
\]
and $\widehat\xi_j=(4\pi/\omega^2)\widehat c_j \exp[-\ii\omega(\widehat t_{1,j},\widehat t_{2,j},0)^T\cdot\hat{\vect d}_0]$.

\section{Aperture-pencil recovery and localization theory}
\label{sec:conditioning}

The preceding construction separates the inverse problem into Hankel completion on a selected aperture pencil and continuous localization from its signal subspaces. We now quantify the sampling requirement in terms of the source--pencil conditioning and deterministic pencil geometry, and then propagate the completion error to the MUSIC locations.

\subsection{Recovery parameters}\label{sec:conditioning-assumptions}

The physical observation set in \eqref{eq:partial-aperture-observations} may be arbitrary. The probabilistic analysis uses independent sampling on the selected sum set.
\begin{assumption}\label{ass:sampling-model}
Let $S=A+B\subset S_{\rm ap}$, $N=|S|\ge2$, and $\Omega=\{k\in S:\zeta_k=1\}$, where the variables $\zeta_k$ are independent Bernoulli variables with parameter $p\in(0,1]$.
\end{assumption}

The source-dependent parameter is the normalized inverse lower frame bound of the two Vandermonde factors.
\begin{definition}\label{def:pencil-frame}
If $\tens V_A$ and $\tens V_B$ have full column rank, define
\begin{equation}\label{eq:pencil-frame-new}
 \mu_1=\max\left\{
 \frac{|A|}{\lambda_{\min}(\tens V_A^*\tens V_A)},
 \frac{|B|}{\lambda_{\min}(\tens V_B^*\tens V_B)}\right\}.
\end{equation}
\end{definition}
For contiguous Fourier samples, \cite{Liao2015,KunisNagelStrotmann2022} obtained Vandermonde lower bounds under a minimum-separation condition. In the following proposition, we use an extended separation condition for general-domain pencil $(A,B)$  containing rectangles inside. 
\begin{proposition}\label{prop:separationcondition}
Suppose that $A$ and $B$ contain Cartesian rectangles $J_A=J_{A,1}\times J_{A,2}$ and $J_B=J_{B,1}\times J_{B,2}$ with contiguous integer intervals. Set $L_A=\min_i|J_{A,i}|$ and $L_B=\min_i|J_{B,i}|$, and assume $L_A,L_B\ge2$. There is a numerical constant $C_{\rm I}>0$ such that, if the locations are $\Delta$-separated and
\[
 \Delta\ge\frac{C_{\rm I}}{\min\{\Delta q_1,\Delta q_2\}}
 \max\{L_A^{-1},L_B^{-1}\},
\]
then both Vandermonde factors have full column rank and
\[
 \mu_1\le C_{\rm I}
 \max\left\{\frac{|A|}{L_A^2},\frac{|B|}{L_B^2}\right\}.
\]
\end{proposition}

Two deterministic quantities describe the pencil geometry. The balance factor is
\begin{equation}\label{eq:shape-factor}
 c_{A,B}=\max\left\{\frac{N}{|A|},\frac{N}{|B|}\right\}
 =\frac{N}{\min\{|A|,|B|\}}.
\end{equation}
For $k\in S$, let $\vect E_k$ denote the $k$th coordinate vector in $\mathbb C^S$, and set $w(k)=\#\{(\alpha,\beta)\in A\times B:\alpha+\beta=k\}$ and $\tens G_k=\calH_{A,B}(\vect E_k)/\sqrt{w(k)}$. For $\tens M\in\mathbb C^{|A|\times|B|}$, define
\begin{align}\label{eq:fiber-energy-norms}
 \|\tens M\|_{\calG,2}^2
 &=\sum_{k\in S}\frac{|\inner{\tens G_k}{\tens M}|^2}{w(k)},\\
 \|\tens M\|_{\infty,2}^2
 &=\max\left\{\max_{\alpha\in A}\|\vect e_\alpha^*\tens M\|_2^2,
 \max_{\beta\in B}\|\tens M\vect e_\beta\|_2^2\right\}.
\end{align}
The quantity $\|\tens M\|_{\calG,2}$ is a seminorm on the ambient matrix space. The fiber parameter is
\begin{equation}\label{eq:fiber-parameter}
 \Gamma_{A,B}=\sup_{\tens M\ne0}
 \frac{\|\tens M\|_{\calG,2}^2}{\|\tens M\|_{\infty,2}^2}.
\end{equation}
Since $\|\calH_{A,B}(\vect E_k)\|_{\calG,2} =\|\calH_{A,B}(\vect E_k)\|_{\infty,2}=1$ for every $k\in S$, we have $\Gamma_{A,B}\ge1$. It measures the accumulation of matrix energy along low-multiplicity fibers. Verifiable geometric conditions for the bound of $\Gamma_{A,B}$ and explicit aperture constructions are given in Section~\ref{sec:verifiableconditions}. In particular, for size-balanced pencils admitting a decomposition into a fixed number of uniformly scale-comparable rectangles, one has $c_{A,B}=O(1)$ and $\Gamma_{A,B}=O(\log^2 N)$.

\subsection{Completion and localization guarantees}
\label{sec:completion-localization-guarantees}

With the sampling model and the source and geometry parameters in place, we first establish exact and stable completion and then propagate the completed-Hankel error to the MUSIC locations.

\begin{theorem}[Exact completion]\label{thm:aperture-pencil-global}
Let $\vect x$ be the $s$-term data in \eqref{eq:X-expsum-new}, with nonzero amplitudes, and suppose that Assumption~\ref{ass:sampling-model} holds and $\mu_1$ is defined. There is a numerical constant $C_2>0$ such that, if
\begin{equation}\label{eq:aperture-global-sample}
 p\ge\frac{C_2\Gamma_{A,B}\mu_1c_{A,B}s\log^2N}{N},
\end{equation}
then $\vect x|_S$ is the unique minimizer of \eqref{eq:admissible-emac-program-new} with probability at least $1-N^{-10}$.
\end{theorem}
Thus, a sufficient number of observed coefficients is of order $\mathcal{O}(\Gamma_{A,B}\mu_1c_{A,B}s\log^2N)$ with high probability\cite{CaiCaiYou2023}. For the decomposable pencils considered below in \Cref{sec:verifiableconditions}, $\Gamma_{A,B}=O(\log^2N)$. The same sampling and structural conditions also yield stability under bounded measurement noise.

\begin{theorem}[Stable completion]\label{thm:noisy-aperture-completion}
Under the hypotheses of Theorem~\ref{thm:aperture-pencil-global}, assume \eqref{eq:aperture-global-sample} and let $\widehat{\vect x}$ be a minimizer of \eqref{eq:physical-noisy-emac-program-new}. There is a numerical constant $C_3>0$ such that, with probability at least $1-N^{-10}$,
\begin{equation}\label{eq:noisy-hankel-bound}
 \|\calH_{A,B}(\widehat{\vect x}-\vect x)\|_2
 \le\frac{C_3}{p}
 \sqrt{s\min\{|A|,|B|\}\log N}\,\eps.
\end{equation}
\end{theorem}

The spectral-norm estimate in \eqref{eq:noisy-hankel-bound} directly controls the perturbation of the left and right signal subspaces used by MUSIC. Let $\calR$ denote the exact counterpart of \eqref{eq:music-residual}, obtained by replacing the estimated signal subspaces with $\range(\tens V_A)$ and $\range(\overline{\tens V_B})$. By \eqref{eq:pencil-factor-new}, $\calR(\vect t_j)=0$ at every true location. For $\calT={\vect t_1,\ldots,\vect t_s}\subset\mathcal F$, assume that
\begin{equation}\label{eq:music-resolution-bound}
\calR(\vect t)\ge c_{\rm M}
\min{\Delta/4,d_{\rm per}(\vect t,\calT)},
\qquad \vect t\in\mathcal F,
\end{equation}
for some $c_{\rm M}>0$. This condition ensures isolated MUSIC zeros and linear growth of the residual near the true locations. Since \eqref{eq:music-resolution-bound} remains valid after decreasing $c_{\rm M}$, below $c_{\rm M}$ denotes a sufficiently small admissible choice. The following proposition gives a verifiable sufficient condition for \eqref{eq:music-resolution-bound}.

\begin{proposition}\label{prop:music-resolution}
Suppose that $A$ and $B$ contain Cartesian cores as in Proposition~\ref{prop:separationcondition}, with $L_A,L_B\ge2$, and that $\Delta$ satisfies the separation inequality there. Suppose additionally that either $A$ or $B$ contains a translate of $\{0,\ldots,s\}^2$. Then \eqref{eq:music-resolution-bound} holds uniformly over all $\Delta$-separated sets $\calT\subset\mathcal F$ of cardinality $s$, with $c_{\rm M}>0$ depending only on $A,B,\mathcal F,s,\Delta,\Delta q_1$, and $\Delta q_2$.
\end{proposition}

The next theorem converts a completed-Hankel perturbation into an off-grid localization error.

\begin{theorem}[Localization stability]\label{thm:completion-localization}
Suppose that $\mu_1$ is defined, the locations are $\Delta$-separated, and \eqref{eq:music-resolution-bound} holds. For a sufficiently small admissible choice of $c_{\rm M}$, there are numerical constants $C_4,C_5>0$ such that, if
\begin{equation}\label{eq:stablecondition}
 \Delta\ge C_4
 \frac{\mu_1\|\calH_{A,B}(\widehat{\vect x}-\vect x)\|_2}
 {c_{\rm M}\sqrt{|A||B|}\min_j|\widetilde c_j|},
\end{equation}
then the ideal continuous minima selected in Algorithm~\ref{alg:general-domain-music} can be labeled so that
\begin{equation}\label{eq:completion-location-final}
 \max_j d_{\rm per}(\widehat{\vect t}_j,\vect t_j)
 \le C_5
 \frac{\mu_1\|\calH_{A,B}(\widehat{\vect x}-\vect x)\|_2}
 {c_{\rm M}\sqrt{|A||B|}\min_j|\widetilde c_j|}.
\end{equation}
\end{theorem}

Combining the completion and localization perturbation bounds yields the following localization error bound.
\begin{corollary}\label{cor:end-to-end-localization}
Under the hypotheses of Theorems~\ref{thm:noisy-aperture-completion} and~\ref{thm:completion-localization}, if
\[
 \Delta\ge C_3C_4
 \frac{\mu_1\sqrt{s\log N}}
 {p c_{\rm M}\sqrt{\max\{|A|,|B|\}}\min_j|\widetilde c_j|}\,\eps,
\]
then, with probability at least $1-N^{-10}$,
\[
 \max_j d_{\rm per}(\widehat{\vect t}_j,\vect t_j)
 \le C_3C_5
 \frac{\mu_1\sqrt{s\log N}}
 {p c_{\rm M}\sqrt{\max\{|A|,|B|\}}\min_j|\widetilde c_j|}\,\eps.
\]
\end{corollary}

\subsection{Verifiable geometry and pencil design}
\label{sec:verifiableconditions}

The preceding results isolate the source-dependent quantity $\mu_1$ from the deterministic pencil factors $c_{A,B}$ and $\Gamma_{A,B}$. Proposition~\ref{prop:separationcondition} controls $\mu_1$ from physical separation; the following decomposition bounds $\Gamma_{A,B}$ using only the finite index sets.
\begin{proposition}\label{prop:rectangular-decomposition}
Suppose that $A=\mathop{\dot\bigcup}_{u=1}^{R_A}A_u$ and $B=\mathop{\dot\bigcup}_{v=1}^{R_B}B_v$, where $A_u=I_{u,1}\times I_{u,2}$ and $B_v=J_{v,1}\times J_{v,2}$ are Cartesian rectangles of integer intervals. Define
\[
 c_{u,v}=\max_{i=1,2}\max\left\{
 \frac{|I_{u,i}+J_{v,i}|}{|I_{u,i}|},
 \frac{|I_{u,i}+J_{v,i}|}{|J_{v,i}|}\right\}.
\]
Then
\[
 \Gamma_{A,B}\le C_1\log^2(eN)
 \sum_{u=1}^{R_A}\sum_{v=1}^{R_B}c_{u,v}^2
\]
for a numerical constant $C_1>0$.
\end{proposition}
The factors $c_{u,v}$ quantify the coordinatewise scale mismatch between the rectangular pieces. In particular, a fixed number of uniformly scale-comparable pieces gives $\Gamma_{A,B}=O(\log^2N)$.

\begin{corollary}\label{cor:rect-decomp-global}
Suppose that Assumption~\ref{ass:sampling-model} holds, the amplitudes are nonzero, $\mu_1$ is defined, and $A,B$ admit the decomposition in Proposition~\ref{prop:rectangular-decomposition}. There is a numerical constant, still denoted by $C_2$, such that exact completion holds with probability at least $1-N^{-10}$ if $p\ge\frac{C_2\mu_1c_{A,B}s\log^4N}{N} \sum_{u=1}^{R_A}\sum_{v=1}^{R_B}c_{u,v}^2$. For a fixed number of uniformly scale-comparable pieces, an sample count  $\mathcal{O}(\mu_1c_{A,B}s\log^4N)$ is therefore sufficient with high probability.
\end{corollary}

The following examples illustrate balanced aperture-admissible pencils and the resulting control of $c_{A,B}$ and $\Gamma_{A,B}$ for representative aperture geometries. The same constructions are used in the numerical experiments. For integers $a\le b$, write $[a,b]_{\mathbb Z}:=[a,b]\cap\mathbb Z$.
\begin{example}\label{ex:balanced-aperture-pencils}
\emph{(i) Elliptical aperture (\Cref{fig:ellipse-music}).} For
$$
S_{\rm ap}=\left\{k\in\mathbb Z^2:
\frac{k_1^2}{14^2}+\frac{k_2^2}{10^2}\le1\right\},
$$
we take
$$
A=\bigl([-6,6]_{\mathbb Z}\times[-1,1]_{\mathbb Z}\bigr)\cup
\bigl([-4,4]_{\mathbb Z}\times[2,4]_{\mathbb Z}\bigr),~
B=\bigl([-6,6]_{\mathbb Z}\times[-1,1]_{\mathbb Z}\bigr)\cup
\bigl([-4,4]_{\mathbb Z}\times[-4,-2]_{\mathbb Z}\bigr).
$$
Then $A+B\subset S_{\rm ap}$, and Proposition~\ref{prop:rectangular-decomposition} gives $c_{A,B}=251/66<4$ and $\Gamma_{A,B}\le19C_1\log^2(eN)$.

\emph{(ii) Sector aperture (\Cref{fig:sector-music}).} For
$$
S_{\rm ap}=\left\{k\in\mathbb Z^2:
k_1^2+k_2^2\le14^2,\ |k_2|\le\sqrt3\,k_1\right\},
$$
we take
$$
A=B=\bigl([3,4]_{\mathbb Z}\times[-5,5]_{\mathbb Z}\bigr)\cup
\bigl([5,6]_{\mathbb Z}\times[-3,3]_{\mathbb Z}\bigr).
$$
Then $A+B\subset S_{\rm ap}$, and Proposition~\ref{prop:rectangular-decomposition} gives $c_{A,B}=41/12<4$ and $\Gamma_{A,B}\le19C_1\log^2(eN)$.

\emph{(iii) Pencil balance (\Cref{fig:pencil-balance-music}).} Let $S_{\rm ap}=[0,20]_{\mathbb Z}^2$ and compare
$$
A_{\rm bal}=B_{\rm bal}=[0,10]_{\mathbb Z}^2,\qquad
A_{\rm un}=[0,5]_{\mathbb Z}^2,\qquad
B_{\rm un}=[0,15]_{\mathbb Z}^2.
$$
Both pairs generate $S_{\rm ap}$. Their balance factors are $c_{\rm bal}=\frac{441}{121}$, $c_{\rm un}=\frac{441}{36}$, while Proposition~\ref{prop:rectangular-decomposition} gives $\Gamma_{\rm bal}\le\frac{441}{121}C_1\log^2(eN)$, $\Gamma_{\rm un}\le\frac{441}{36}C_1\log^2(eN)$. Thus the balanced factorization has the smaller sufficient geometric factor in Corollary~\ref{cor:rect-decomp-global}.

\emph{(iv) Four-window aperture (\Cref{fig:four-window-music}).} Let $P=\{0,1\}^2$, $Q=\{0,1\}\times\{0,1,2\}$,  $C=P+Q$, and
$$
S_{\rm ap}=\bigcup_{a,b=0}^1\bigl(C+18a\vect e_1+16b\vect e_2\bigr).
$$
We take
$$
A=P\cup(P+18\vect e_1),\qquad B=Q\cup(Q+16\vect e_2).
$$
Then $A+B=S_{\rm ap}$, with $c_{A,B}=6$ and $\Gamma_{A,B}\le16C_1\log^2(eN)$.
\end{example}

\section{Numerical experiments}
\label{sec:experiments}
We evaluate the complete reconstruction procedure consisting of Algorithm~\ref{alg:admm-hankel-completion} followed by Algorithm~\ref{alg:general-domain-music}. The experiments first illustrate reconstruction on connected and disconnected physical apertures, then examine the effect of pencil geometry predicted by our theory, and finally study noise sensitivity and computational time. All computations were performed in MATLAB R2024b on a 64-bit Windows 11 laptop equipped with an Intel Core Ultra 5 125H processor.

Unless stated otherwise, synthetic data are generated from the finite-radius first-Born field and rescaled according to \eqref{eq:finite-radius-demodulated}, with $\omega=3$, $\Delta q_1=\Delta q_2=2\pi/80$, $q_{1,0}=q_{2,0}=0$, $\hat{\vect d}_0=(0,0,1)^T$, and $\rho(\hat{\vect r})\equiv R_{\rm obs}=10^5$. The off-grid locations are drawn from $[0.15,0.45]^2$ with minimum periodic separation $0.075$, while the scattering strengths have independent phases uniformly distributed on $[0,2\pi)$ and magnitudes uniformly distributed on $[0.03,0.05]$. For a prescribed sample budget $m$, the observed indices are drawn uniformly without replacement from the admissible aperture. Additive complex Gaussian noise $\bm\nu$ is scaled so that $\mathrm{SNR}_{\rm dB}=20\log_{10}\bigl(\|(F_{\rm fr}(\hat{\vect r}_k))_{k\in\Omega}\|_2/\|(\nu_k)_{k\in\Omega}\|_2\bigr)$. The localization error is the matched root-mean-square distance normalized by the diameter of the field of view, and a trial is declared successful when this relative error is at most $5\%$. ADMM is run for at most $300$ iterations with adaptive $\tau\in[10^{-6},10^6]$ and feasibility tolerance $10^{-9}$. We use $\tau_0=1$, residual-balance parameter $10$, update interval $10$, and absolute and relative tolerances $10^{-7}$ for the experiments in \Cref{fig:four-window-music,fig:nine-window-music}, and $\tau_0=4$, residual-balance parameter $2$, update interval $5$, and absolute and relative tolerances $10^{-4}$ for the connected-aperture experiments and those in \Cref{fig:pencil-balance-music,fig:snr-runtime-music}.

\subsection{Localization on connected apertures}

We first illustrate the complete reconstruction procedure on two connected apertures with different geometries. The examples demonstrate how aperture-admissible pencils can be adapted to the acquisition geometry and used for off-grid localization from sparse observations.

For the sector aperture in \Cref{fig:sector-music}, we use the balanced pencil
$$
A=B=\bigl([3,4]\times[-5,5]\bigr)\cup\bigl([5,6]\times[-3,3]\bigr),
$$
for which $A+B$ is contained in the accessible aperture. We observe $m=13$ of the $123$ admissible coefficients at $40$ dB. As shown in \Cref{fig:sector-music}, the two scatterers are accurately localized.

\begin{figure}[t]
\centering
\includegraphics[width=\textwidth]{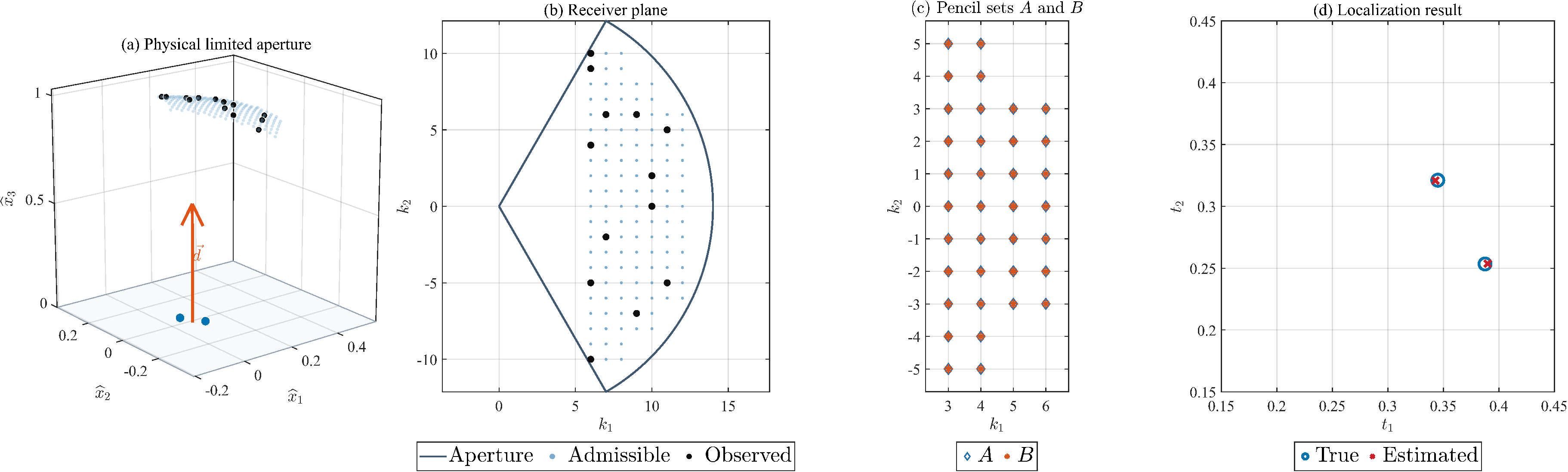}
\caption{Sector aperture with two scatterers reconstructed from $m=13$ of $123$ admissible coefficients at $40$ dB. The panels show the acquisition geometry, observed indices, pencil factors, and reconstructed locations.}
\label{fig:sector-music}
\end{figure}

For the elliptical aperture in \Cref{fig:ellipse-music}, we take
$$
\begin{aligned}
A&=\bigl([-6,6]\times[-1,1]\bigr)\cup\bigl([-4,4]\times[2,4]\bigr),\\
B&=\bigl([-6,6]\times[-1,1]\bigr)\cup\bigl([-4,4]\times[-4,-2]\bigr).
\end{aligned}
$$
 Using $m=46$ observations at $40$ dB, the five scatterers, including a prescribed close pair, are accurately localized; see \Cref{fig:ellipse-music}.

\begin{figure}[t]
\centering
\includegraphics[width=\textwidth]{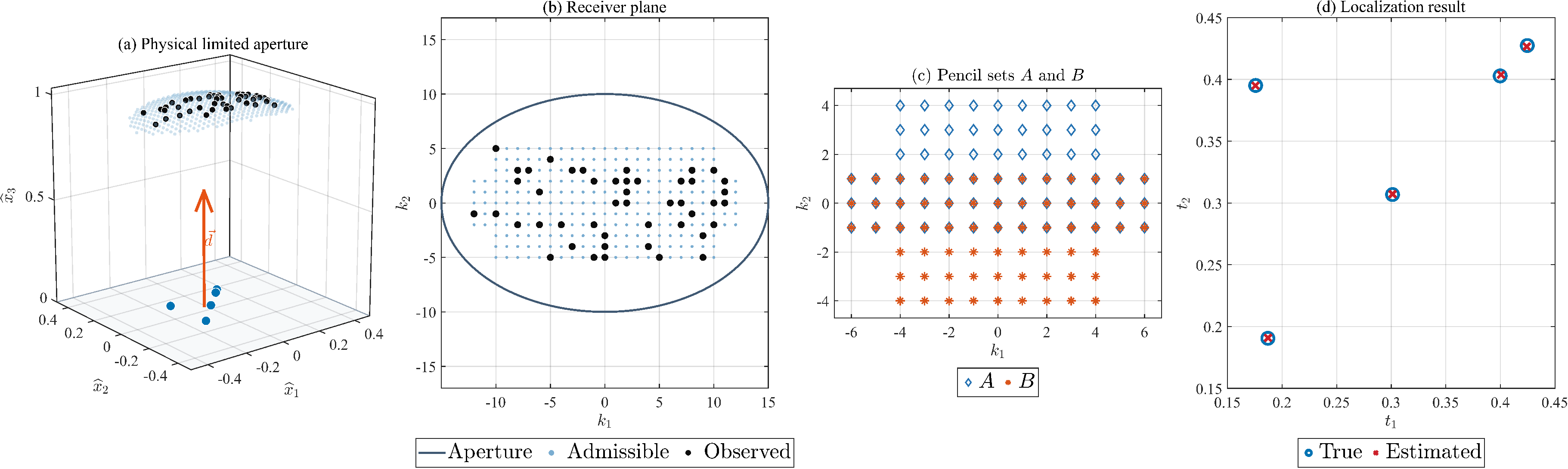}
\caption{Elliptical aperture with five scatterers reconstructed from $m=46$ of $251$ admissible coefficients at $40$ dB. The panels show the acquisition geometry, observed indices, pencil factors, and reconstructed locations.}
\label{fig:ellipse-music}
\end{figure}

\subsection{Disconnected acquisition windows}

We next consider disconnected acquisition geometries consisting of several separated receiver windows. For four $3\times4$ windows with shifts of $18$ and $16$ lattice units in the two coordinate directions, we use
$$
A=\{0,1\}^2\cup\bigl(\{0,1\}^2+18\vect e_1\bigr),\qquad B=\bigl(\{0,1\}\times\{0,1,2\}\bigr)\cup\bigl(\{0,1\}\times\{0,1,2\}+16\vect e_2\bigr).
$$
Then $A+B$ is exactly the union of four translated copies of $\{0,1,2\}\times\{0,1,2,3\}$. Figure~\ref{fig:four-window-music} shows the corresponding acquisition geometry, sparse observations, pencil factors, and localization result, where the reconstruction is from $m=32$ observations at $30$ dB.

\begin{figure}[htbp]
\centering
\includegraphics[width=\textwidth]{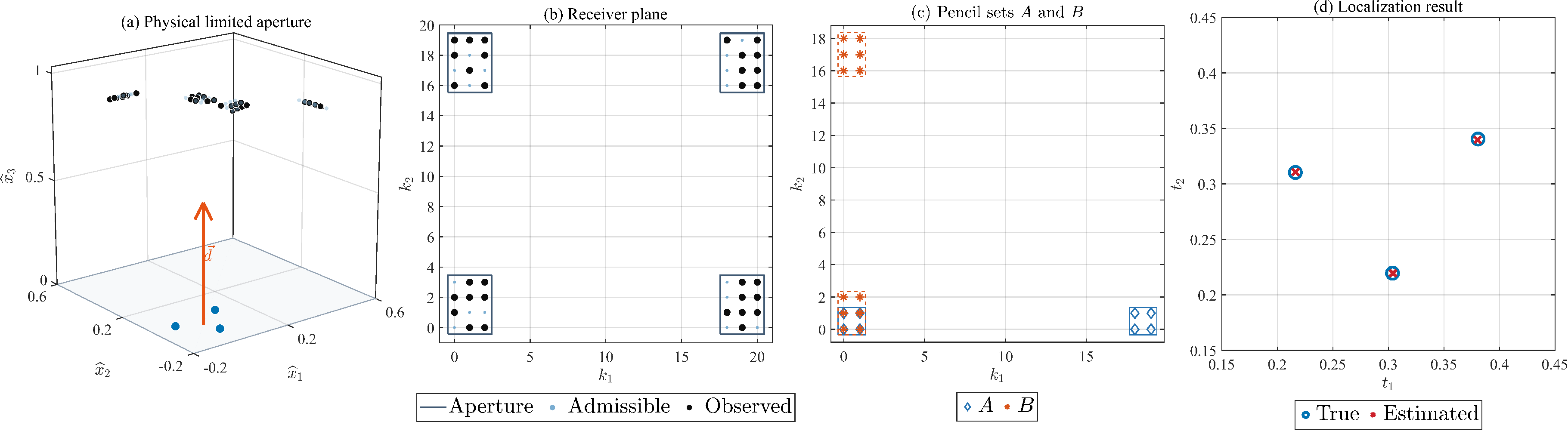}
\caption{Four separated $3\times4$ acquisition windows: acquisition geometry, observed indices, balanced pencil factors, and localization of three scatterers from $m=32$ observations at $30$ dB.}
\label{fig:four-window-music}
\end{figure}

We further consider nine separated $3\times3$ windows. Let $Q=\{0,1\}^2$, $\Lambda=\{0,8\}\times\{0,10\}$, and $\vect\ell=(9,11)^T$, and take
$$
A=Q+\Lambda+\vect\ell,\qquad B=Q+\Lambda-\vect\ell.
$$
It follows that
$$
A+B=\bigcup_{a,b=0}^2\left(\{0,1,2\}^2+8a\vect e_1+10b\vect e_2\right),
$$
which consists of nine separated $3\times3$ windows. Figure~\ref{fig:nine-window-music} shows the reconstruction from $m=33$ observations at $30$ dB.

\begin{figure}[htbp]
\centering
\includegraphics[width=\textwidth]{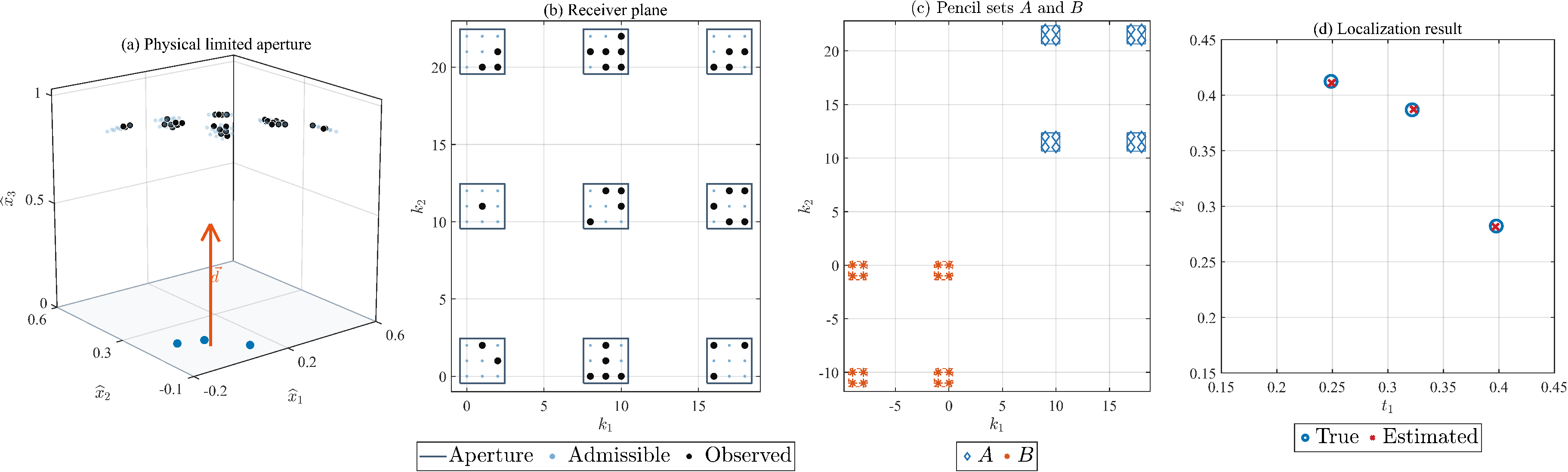}
\caption{Nine separated $3\times3$ acquisition windows: acquisition geometry, observed indices, balanced pencil factors, and localization of three scatterers from $m=33$ observations at $30$ dB.}
\label{fig:nine-window-music}
\end{figure}

These examples illustrate that the proposed pencil construction naturally accommodates disconnected acquisition geometries.

\subsection{Pencil balance, noise stability, and computational time}

We first examine the effect of pencil balance while keeping the sum set and observations fixed. Let $S=[0,20]_{\mathbb Z}^2$ and compare
\[
A_{\rm bal}=B_{\rm bal}=[0,10]_{\mathbb Z}^2,\qquad A_{\rm un}=[0,5]_{\mathbb Z}^2,\qquad B_{\rm un}=[0,15]_{\mathbb Z}^2.
\]
Both pencils generate the same sum set with $N=441$, while their balance factors are $c_{\rm bal}=441/121$ and $c_{\rm un}=441/36$, respectively. We fix $s=3$ and the SNR at $40$ dB, and take $m\in\{12,15,18,21,24,27,30,32,35,38,41\}$. For each $m$, the two pencils are evaluated on $30$ paired trials using the same source realization, observations, and noise. As shown in \Cref{fig:pencil-balance-music}, the balanced pencil enters the reliable recovery regime with fewer observations and yields smaller localization errors throughout the transition region. 

\begin{figure}[t]
\centering
\includegraphics[width=0.465\textwidth]{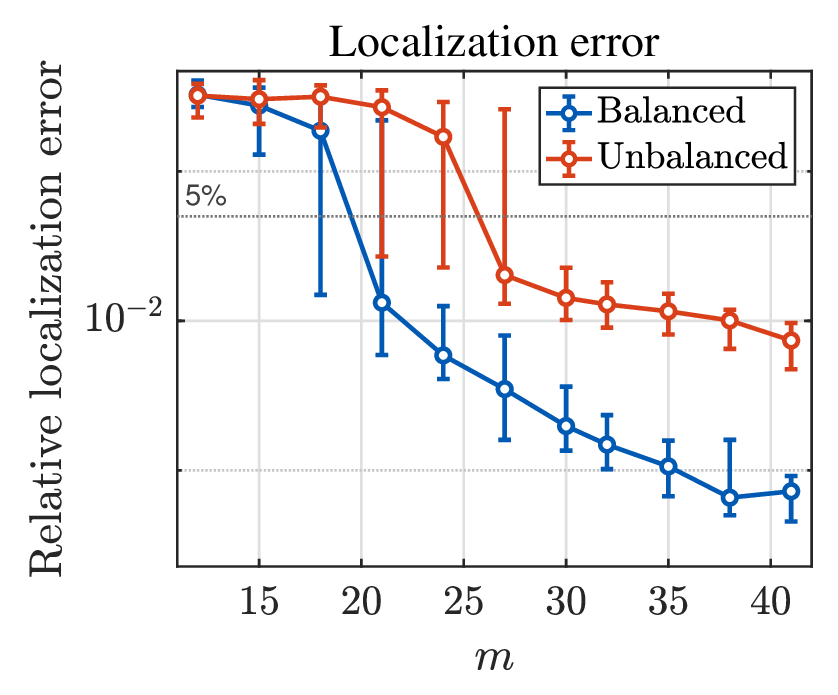}\hfill
\includegraphics[width=0.465\textwidth]{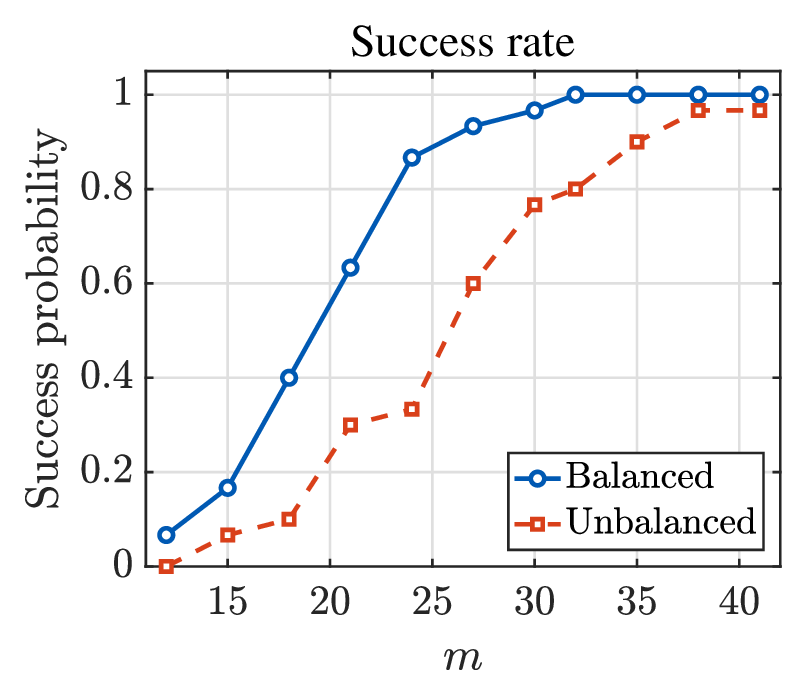}
\caption{Comparison of balanced and unbalanced pencils with the same sum set. Each point is based on $30$ paired trials at $40$ dB.}
\label{fig:pencil-balance-music}
\end{figure}

We next study noise sensitivity and computational time on a family of disconnected four-window apertures. For $C_L=[0,L]_{\mathbb Z}^2$, define
\[
A_L=C_L\cup(C_L+18L\vect e_1),\qquad B_L=C_L\cup(C_L+16L\vect e_2).
\]
Then $S_L=A_L+B_L$ consists of four separated $(2L+1)\times(2L+1)$ windows and $N=|S_L|=4(2L+1)^2$. We set $\Delta q_1=\Delta q_2=2\pi/(80L)$ so that the physical acquisition geometry remains fixed as $L$ varies. For the noise experiment, we take $L=6$, so that $N=676$, and use $m=\lceil6s\log N\rceil=118$. The SNR varies from $10$ to $40$ dB in increments of $5$ dB, with $30$ paired trials at each noise level. For the timing experiment, we take $L\in\{1,2,3,4,5,6\}$ and $m=\min\{N,\lceil6s\log N\rceil\}$ at $40$ dB. For each problem size, the reported time is the median over $30$ trials for ADMM completion and MUSIC evaluation. As shown in left side of \Cref{fig:snr-runtime-music}, the localization error decreases steadily as the observation SNR increases. Right side of \Cref{fig:snr-runtime-music} shows the increase in computational time with the sum-set size under the same sampling rule.

\begin{figure}[t]
\centering
\includegraphics[width=0.465\textwidth]{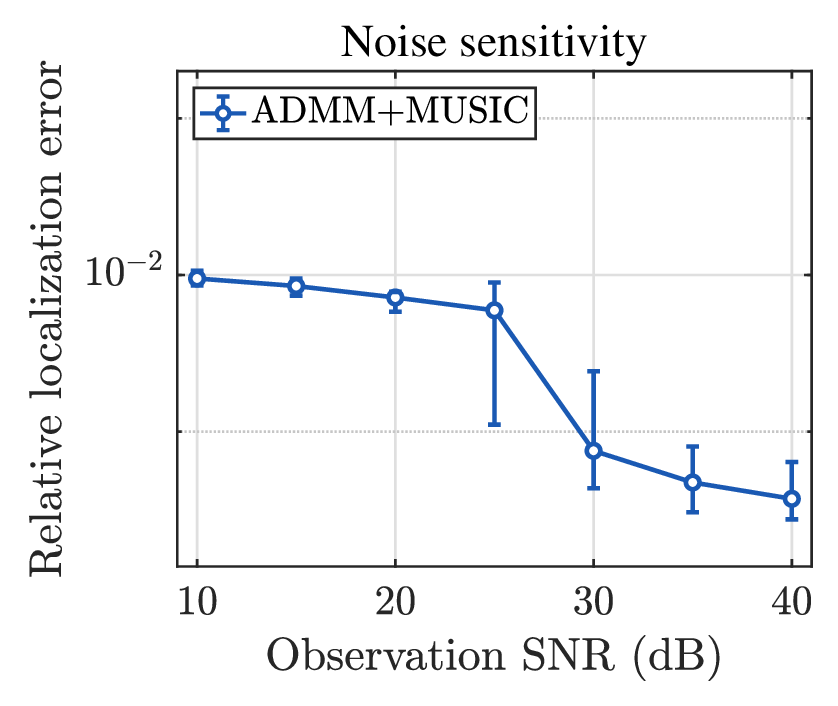}\hfill
\includegraphics[width=0.465\textwidth]{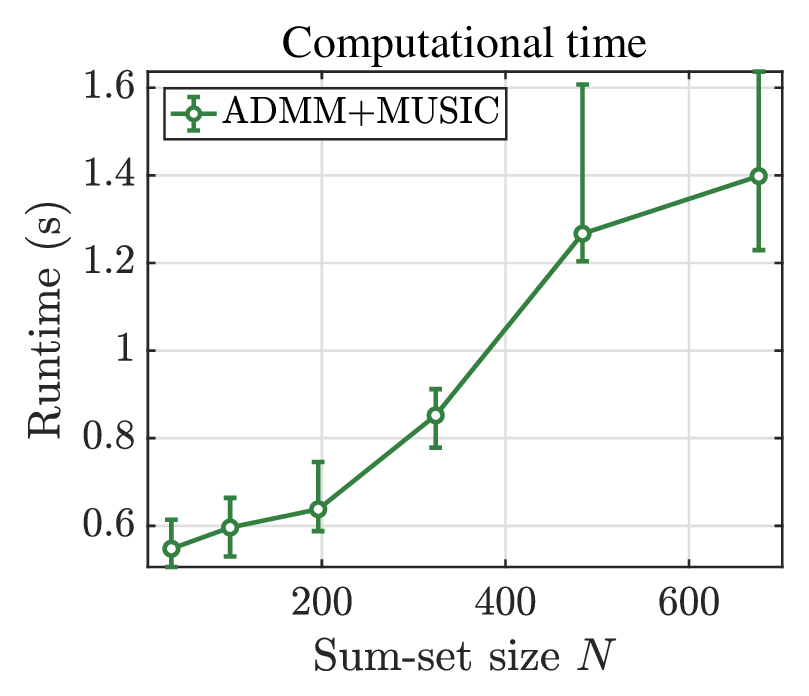}
\caption{Noise sensitivity and computational time for the four-window acquisition geometry. Error bars indicate interquartile ranges over $30$ trials.}
\label{fig:snr-runtime-music}
\end{figure}

\section{Concluding remarks}\label{sec:conclusion}

We developed a geometry adapted Hankel framework for localizing finite planar point scatterers from sparse limited-aperture Born data. The main contribution is to incorporate the acquisition geometry directly into the Hankel pencil design and to quantify the resulting recovery guarantee through Vandermonde conditioning, pencil balance, and fiber accumulation. This yields a principled criterion for selecting aperture-admissible pencils, together with exact and stable completion guarantees and a subsequent MUSIC localization bound. Numerical experiments further illustrate the advantage of the proposed design on nonrectangular and disconnected apertures. Extensions to deterministic sampling and more general scattering models are left for future work.

\section*{Declaration on the use of artificial intelligence}
Generative AI tools were used during revision to assist with language editing, structural reorganization, bibliographic cross-checking, and consistency checks of notation, formulas, and proofs. The authors assume responsibility for all content.

\appendix

\section{Proofs}
\label{app:proofs-admissible}

The Frobenius inner product is $\inner{\tens X}{\tens Y}=\tr(\tens X^*\tens Y)$. For matrices $\tens X,\tens Y$, the rank-one operator $\tens X\otimes\tens Y$ is defined by $(\tens X\otimes\tens Y)\tens M=\inner{\tens Y}{\tens M}\tens X$. For $k\in S$, write
\[
        \tens{P}_k:=\calH_{A,B}(\vect{E}_k)
        =\sqrt{w(k)}\,\tens{G}_k.
\]
Every $\tens{P}_k$ is a $0$-$1$ partial permutation matrix. Hence
\[
\norm{\tens{P}_k}_2=1,\qquad
\norm{\tens{G}_k}_F=1,\qquad
\norm{\tens{G}_k}_2=w(k)^{-1/2}.
\]
The atoms $\{\tens{G}_k:k\in S\}$ have disjoint supports and therefore form an orthonormal basis of $\HankelSpace_{A,B}=\range(\calH_{A,B})$.  We use
\[
        \calA:=\sum_{k\in S}\tens{G}_k\otimes\tens{G}_k
\]
for the orthogonal projector onto this subspace and $\calA^\perp:=\calI-\calA$. Let $\tens{M}_\star=\tens{U}\tens{\Sigma}\tens{V}^*$ be a compact SVD, and define its rank-$s$ tangent space by
\[
        T:=\{\tens{U}\tens{R}^*+\tens{L}\tens{V}^*:
        \tens{R}\in\C^{|B|\times s},\ 
        \tens{L}\in\C^{|A|\times s}\}.
\]
The orthogonal projection onto the tangent space is
\begin{equation}\label{eq:tangent-projection-formula-app}
        \calP_T(\tens{M})
        =\tens{U}\tens{U}^*\tens{M}
        +\tens{M}\tens{V}\tens{V}^*
        -\tens{U}\tens{U}^*\tens{M}\tens{V}\tens{V}^*.
\end{equation}

\subsection{Technical lemmas}
\label{subsec:technical-lemmas}
The following estimates are proved in Section~\ref{subsec:proofs-technical-lemmas}. 
\begin{lemma}\label{lem:leverage-frame}
 With $\mu_1$ defined in Definition~\ref{def:pencil-frame}, the singular-subspace incoherence holds
\begin{equation}\label{eq:leverage-frame-new}
\max_{\alpha\in A}\|\tens U^*\vect e_\alpha\|_2^2\le\frac{\mu_1s}{|A|},\qquad \max_{\beta\in B}\|\tens V^*\vect e_\beta\|_2^2\le\frac{\mu_1s}{|B|}.
\end{equation}
\end{lemma}

\begin{lemma}
\label{lem:rectangular-fiber-bound}
Let $A_0=I_1\times I_2$ and $B_0=J_1\times J_2$ be Cartesian rectangles.  Assume that, for $i=1,2$,
\[
        \max\left\{
        \frac{|I_i+J_i|}{|I_i|},
        \frac{|I_i+J_i|}{|J_i|}
        \right\}\le c_{\rm rec}.
\]
Then every $\tens{M}\in\C^{|A_0|\times|B_0|}$ satisfies
\begin{equation}\label{eq:rect-fiber-energy}
        \norm{\tens{M}}_{\calG,2}^2
        \le 576c_{\rm rec}^2
        \bigl(1+\log|A_0+B_0|\bigr)^2
        \norm{\tens{M}}_{\infty,2}^2,
\end{equation}
where the $\calG$-seminorm is formed with the fibers of the rectangle pair.
\end{lemma}

\begin{lemma}
\label{lem:kernel-bounds}
With $\mu_1$ defined in Definition~\ref{def:pencil-frame}, we have for every $a,b\in S$,
\begin{equation}\label{eq:tangent-atom-bound-app}
        \max_{k\in S}\norm{\calP_T\tens{G}_k}_F^2
        \le \frac{2\mu_1c_{A,B}s}{N}, \quad
|\inner{\tens{G}_b}{\calP_T\tens{G}_a}|
        \le\frac{3\mu_1c_{A,B}s}{N}
        \sqrt{\frac{w(b)}{w(a)}}.
\end{equation}
\end{lemma}

\begin{lemma}
\label{lem:initial-structural-norms}
Let $\tens{Q}_\star=\tens{U}\tens{V}^*$, and define $\norm{\tens M}_{\calG,\infty} :=\max_{k\in S}\frac{|\inner{\tens G_k}{\tens M}|}{\sqrt{w(k)}}$. Then, for every $a\in S$, we have
\begin{equation}\label{eq:Q-app}
\norm{\tens{Q}_\star}_{\calG,\infty}\le\frac{\mu_1c_{A,B}s}{N},\quad
\norm{\tens{Q}_\star}_{\calG,2}^2\le\frac{\mu_1c_{A,B}\Gamma_{A,B}s}{N},\quad
\norm{\calP_T\tens{P}_a}_{\calG,2}^2
        \le \frac{10\mu_1c_{A,B}\Gamma_{A,B}s}{N}
\end{equation}

\end{lemma}
\begin{lemma}
\label{lem:completion-certificate-event}
Let $0<\delta<1/2$ and suppose that $p\ge \frac{C_0\mu_1c_{A,B}\Gamma_{A,B}s\log^2(N/\delta)}{N}$, where $C_0>0$ is a sufficiently large numerical constant.  If $\Omega\subseteq S$ is drawn by independent Bernoulli-$p$ sampling and $\calA_\Omega:=\sum_{k\in\Omega}\tens G_k\otimes\tens G_k $, then, with probability at least $1-\delta$, the following statements hold:
\begin{equation}\label{eq:full-sample-tangent-event-app}
        \norm{\calP_T(p^{-1}\calA_\Omega-\calA)\calP_T}\le\frac12,
\end{equation}
there exists $\tens W\in\range(\calA_\Omega)+\range(\calA^\perp)$ such that
\begin{align}
\norm{\calP_T(\tens W-\tens Q_\star)}_F
&\le\frac1{8N},
\qquad
\norm{\calP_{T^\perp}\tens W}_2\le\frac12,
\label{eq:det-certificate-bounds-app}\\
\norm{\calA_\Omega\tens W}_F
&\le C_{\rm W}\sqrt{\frac{s\log(N/\delta)}{p}}.
\label{eq:observed-certificate-bound-app}
\end{align}
where $C_{\rm W}>0$ is a numerical constant.
\end{lemma}

\subsection{Proof of Proposition~\ref{prop:separationcondition}}
\label{subsec:proof-separationcondition}
\begin{proof}
For $s=1$, both Gram matrices equal their respective cardinalities, so $\mu_1=1$. Assume $s\ge2$ and set $\vect f_j=(\Delta q_1t_{1,j},\Delta q_2t_{2,j})^T/(2\pi)$. Their separation on the unit torus satisfies
\[
 q:=\min_{j\ne k}\min_{\vect m\in\mathbb Z^2}\|\vect f_j-\vect f_k-\vect m\|_\infty
 \ge\frac{\min\{\Delta q_1,\Delta q_2\}}{2\pi\sqrt2}\Delta.
\]
For $C\in\{A,B\}$, choose a translate $K_C$ of $\{0,\ldots,n_C-1\}^2$ inside $J_C$, where $n_C=2\lfloor L_C/2\rfloor\ge L_C/2$. The assumed separation implies $qn_C>8$ for $C_{\rm I}$ sufficiently large. Translation of the index set only multiplies the Vandermonde matrix by a diagonal unitary matrix. Therefore, by \cite[Theorem~4.2]{KunisNagelStrotmann2022}, $\tens V_{K_C}$ has full column rank and $\sigma_s(\tens V_{K_C})>0.9n_C$. Adding the remaining rows gives
\[
 \lambda_{\min}(\tens V_C^*\tens V_C)
 \ge\sigma_s(\tens V_{K_C})^2\ge L_C^2/5,
 \qquad C\in\{A,B\}.
\]
Both factors thus have full column rank, and substitution into \eqref{eq:pencil-frame-new} proves the claimed bound after increasing $C_{\rm I}$ if necessary.
\end{proof}

\subsection{Proof of Proposition~\ref{prop:music-resolution}}
\label{subsec:proof-music-resolution}
\begin{proof}
Assume first that $A$ contains a translate of $\{0,\ldots,s\}^2$; the case of $B$ follows by complex conjugation. Translation leaves the corresponding normalized one-sided MUSIC residual unchanged, so we may assume $\{0,\ldots,s\}^2\subset A$. Proposition~\ref{prop:separationcondition} gives uniform full column rank of $\tens V_A$ and $\tens V_B$ over the $\Delta$-separated configurations.

Fix $\calT=\{\vect t_1,\ldots,\vect t_s\}\subset\mathcal F$ and write $\vect z(\vect t)=(\ee^{-\ii\Delta q_1t_1},\ee^{-\ii\Delta q_2t_2})^T$. If $\vect t\notin\calT$, let $\vect z_1,\ldots,\vect z_{s+1}$ denote the $s$ source nodes and $\vect z(\vect t)$. For each $\ell$, choose affine polynomials $\ell_{\ell m}$ satisfying $\ell_{\ell m}(\vect z_m)=0$ and $\ell_{\ell m}(\vect z_\ell)\ne0$, and set
\[
 \pi_\ell(\vect z)=\prod_{m\ne\ell}\frac{\ell_{\ell m}(\vect z)}{\ell_{\ell m}(\vect z_\ell)},\qquad \pi_\ell(\vect z_m)=\delta_{\ell m}.
\]
Each $\pi_\ell$ has total degree at most $s$, hence its coefficient vector is supported on $\{0,\ldots,s\}^2$. Therefore $[\vect v_A(\vect t_1)\ \cdots\ \vect v_A(\vect t_s)\ \vect v_A(\vect t)]$ has rank $s+1$, so $\calR(\vect t)>0$ for every $\vect t\in\mathcal F\setminus\calT$. Fix $j$. As above, choose a polynomial $\psi_j$ of total degree at most $s-1$ such that $\psi_j(\vect z(\vect t_m))=0$ for $m\ne j$ and $\psi_j(\vect z(\vect t_j))\ne0$. Suppose
\[
 \tens V_A\vect a+b_1\partial_{t_1}\vect v_A(\vect t_j)+b_2\partial_{t_2}\vect v_A(\vect t_j)=0.
\]
Let $\vect\gamma_{\ell,j}$ be the coefficient vector of $(z_\ell-z_{\ell,j})\psi_j$, $\ell=1,2$. Since these polynomials vanish at every source node, left multiplication by $\vect\gamma_{\ell,j}^T$ gives
\[
 0=-\ii\Delta q_1z_{1,j}\psi_j(\vect z(\vect t_j))b_1,\qquad
 0=-\ii\Delta q_2z_{2,j}\psi_j(\vect z(\vect t_j))b_2.
\]
Hence $b_1=b_2=0$, and then $\vect a=0$ by the full column rank of $\tens V_A$. Thus $[\tens V_A\ \partial_{t_1}\vect v_A(\vect t_j)\ \partial_{t_2}\vect v_A(\vect t_j)]$ has rank $s+2$.

Set $\vect g_A(\vect t)=|A|^{-1/2}(\tens I-\tens U\tens U^*)\vect v_A(\vect t)$. Since $\ker(\tens I-\tens U\tens U^*)=\range(\tens V_A)$, the preceding rank condition implies that $D\vect g_A(\vect t_j)$ has real rank two. As $\vect g_A(\vect t_j)=0$, Taylor's theorem gives $r_j,\kappa_j>0$ such that $\|\vect g_A(\vect t)\|_2\ge\kappa_jd_{\rm per}(\vect t,\vect t_j)$ whenever $d_{\rm per}(\vect t,\vect t_j)\le r_j$. Since $\calR(\vect t)\ge\|\vect g_A(\vect t)\|_2$, the same local lower bound holds for $\calR$.

For the compactness argument, include configurations on the boundary of $\mathcal F$; the preceding rank arguments remain valid there. The ordered $\Delta$-separated configurations form a compact subset of $\mathcal F^s$. By Proposition~\ref{prop:separationcondition} and continuity, the preceding radii and lower bounds may therefore be chosen uniformly, say $r>0$ and $c>0$, with $r\le\Delta/4$. On the compact set of configuration--test-point pairs satisfying $d_{\rm per}(\vect t,\calT)\ge r$, the first part of the proof gives $\calR(\vect t)>0$, and hence its minimum $\gamma$ is positive. Taking $c_{\rm M}=\min\{c,4\gamma/\Delta\}$ proves \eqref{eq:music-resolution-bound}.
\end{proof}

\subsection{Proof of Proposition~\ref{prop:rectangular-decomposition}}
\label{subsec:proof-rectangular-decomposition}

\begin{proof}
For each pair $(u,v)$, write $w_{u,v}(k):=\#\{(\alpha,\beta)\in A_u\times B_v:\alpha+\beta=k\}$ and $s_{u,v}(k):=\sum_{\alpha\in A_u,\,\beta\in B_v,\,\alpha+\beta=k}M_{\alpha,\beta}$. Since the decompositions of $A$ and $B$ are disjoint, we have
\[
w(k)=\sum_{u,v}w_{u,v}(k),\qquad
\sum_{\alpha+\beta=k}M_{\alpha,\beta}=\sum_{u,v}s_{u,v}(k).
\]
For each fixed $k$, weighted Cauchy--Schwarz gives
\begin{align*}
\frac{\bigl|\sum_{\alpha+\beta=k}M_{\alpha,\beta}\bigr|^2}{w(k)^2}
&=\frac{\bigl|\sum_{u,v}s_{u,v}(k)\bigr|^2}{w(k)^2}
\le \frac1{w(k)}
\sum_{u,v:w_{u,v}(k)>0}\frac{|s_{u,v}(k)|^2}{w_{u,v}(k)}\\
&\le
\sum_{u,v:w_{u,v}(k)>0}\frac{|s_{u,v}(k)|^2}{w_{u,v}(k)^2},
\end{align*}
where the last inequality follows from $w(k)\geq w_{u,v}(k)$. Summing over $k$ gives
\begin{align*}
\norm{\tens M}_{\calG,2}^2
&\le \sum_{u=1}^{R_A}\sum_{v=1}^{R_B}
\norm{\tens M|_{A_u\times B_v}}_{\calG_{u,v},2}^2
\le 576(1+\log N)^2
\sum_{u=1}^{R_A}\sum_{v=1}^{R_B}c_{u,v}^2
\norm{\tens M|_{A_u\times B_v}}_{\infty,2}^2\\
&\le 576\log^2(eN)
\left(\sum_{u=1}^{R_A}\sum_{v=1}^{R_B}c_{u,v}^2\right)
\norm{\tens M}_{\infty,2}^2.
\end{align*}
The second inequality follows from Lemma~\ref{lem:rectangular-fiber-bound}, and the last uses $\norm{\tens M|_{A_u\times B_v}}_{\infty,2}\le\norm{\tens M}_{\infty,2}$. The result follows from \eqref{eq:fiber-parameter}.
\end{proof}

\subsection{Proof of Theorem~\ref{thm:aperture-pencil-global}}
\label{subsec:proof-aperture-pencil-global}
\begin{proof}
The case $p=1$ is immediate. Otherwise, set $\delta=N^{-10}$. Equation~\eqref{eq:aperture-global-sample} and Lemma~\ref{lem:completion-certificate-event} imply that, with probability at least $1-N^{-10}$, event~\eqref{eq:full-sample-tangent-event-app} holds and there is a matrix $\tens W\in\range(\calA_{\Omega})+\range(\calA^\perp)$ satisfying \eqref{eq:det-certificate-bounds-app}. We prove uniqueness on this event. Let $\tens E=\calH_{A,B}(\vect h)$ be any feasible perturbation. Then $\tens E\in\range(\calA)$ and $\calA_{\Omega}\tens E=0$, hence $\calK\tens E=0$, where $\calK=p^{-1}\calA_{\Omega}+\calA^\perp$. By \eqref{eq:full-sample-tangent-event-app},
\[
\frac12\norm{\calP_T\tens E}_F^2
\le \inner{\calP_T\tens E}{\calK\calP_T\tens E}
=-\operatorname{Re}\inner{\calP_T\tens E}{\calK\calP_{T^\perp}\tens E}
\le p^{-1}\norm{\calP_T\tens E}_F\norm{\calP_{T^\perp}\tens E}_F,
\]
and thus $\norm{\calP_T\tens E}_F\le2p^{-1}\norm{\calP_{T^\perp}\tens E}_F$. Choose $\tens F\in T^\perp$, $\norm{\tens F}_2\le1$, such that $\operatorname{Re}\inner{\tens F}{\calP_{T^\perp}\tens E} =\norm{\calP_{T^\perp}\tens E}_*$. Since $\tens Q_\star+\tens F\in\partial\norm{\tens M_\star}_*$ and $\inner{\tens W}{\tens E}=0$, we get
\begin{align*}
\norm{\tens M_\star+\tens E}_*
&\ge \norm{\tens M_\star}_*
+\operatorname{Re}\inner{\tens Q_\star+\tens F}{\tens E}\\
&\ge \norm{\tens M_\star}_*
-\norm{\calP_T(\tens W-\tens Q_\star)}_F\norm{\calP_T\tens E}_F
+\bigl(1-\norm{\calP_{T^\perp}\tens W}_2\bigr)\norm{\calP_{T^\perp}\tens E}_* \\
&\ge \norm{\tens M_\star}_*
+\left(1-\frac12-\frac1{4Np}\right)\norm{\calP_{T^\perp}\tens E}_*
\ge \norm{\tens M_\star}_*+\frac14\norm{\calP_{T^\perp}\tens E}_*,
\end{align*}
where the second inequality follows from $\inner{\tens W}{\tens E}=0$, and the last two inequalities follow from \eqref{eq:det-certificate-bounds-app}, $\norm{\calP_T\tens E}_F\le2p^{-1}\norm{\calP_{T^\perp}\tens E}_F\le 2p^{-1}\norm{\calP_{T^\perp}\tens E}_*$, and the sampling condition, after enlarging $C_2$ if necessary, gives $p\ge N^{-1}$. Thus every feasible perturbation with $\calP_{T^\perp}\tens E\ne0$ strictly increases the objective. If $\calP_{T^\perp}\tens E=0$, then the tangent estimate gives $\tens E=0$. Since $\calH_{A,B}$ is injective on $\C^S$, we have $\vect h=0$. This proves uniqueness.
\end{proof}
\subsection{Proof of Theorem~\ref{thm:noisy-aperture-completion}}
\label{subsec:proof-noisy-aperture-completion}
\begin{proof}
Set $\vect h=\widehat{\vect x}-\vect x$, $\tens E=\calH_{A,B}(\vect h)$, and $w_{\max}=\max_{k\in S}w(k)$. Feasibility and optimality imply
\begin{align}
 \|\calP_{\Omega}\vect h\|_2&\le2\eps,
 \label{eq:noisy-feasibility-2norm}\\
 \|\tens M_\star+\tens E\|_*&\le\|\tens M_\star\|_*.
\end{align}
Consequently,
\begin{equation}\label{eq:noisy-observed-lifted-2norm}
 \|\calA_{\Omega}\tens E\|_F^2
 =\sum_{k\in\Omega}w(k)|h_k|^2
 \le4w_{\max}\eps^2.
\end{equation}
On the event in Lemma~\ref{lem:completion-certificate-event} with $\delta=N^{-10}$, let $\tens W$ be the stated certificate. The nuclear-norm subgradient inequality gives
\[
 0\ge\operatorname{Re}\inner{\tens Q_\star}{\calP_T\tens E}
 +\|\calP_{T^\perp}\tens E\|_*.
\]
Using the tangent and normal parts of $\tens W$ and then \eqref{eq:det-certificate-bounds-app}, we obtain
\begin{align}
 \|\calP_{T^\perp}\tens E\|_*
 &\le |\inner{\tens W}{\tens E}|
 +\frac{1}{8N}\|\calP_T\tens E\|_F
 +\frac12\|\calP_{T^\perp}\tens E\|_*.
 \label{eq:noisy-normal-2norm}
\end{align}
Because $\tens E\in\range(\calA)$ and $\tens W\in\range(\calA_{\Omega})+\range(\calA^\perp)$,
\begin{align*}
 |\inner{\tens W}{\tens E}|
 =|\inner{\calA_{\Omega}\tens W}
 {\calA_{\Omega}\tens E}|
 \le2C_{\rm W}\sqrt{11p^{-1}s w_{\max}\log N}\,\eps,
\end{align*}
where \eqref{eq:noisy-observed-lifted-2norm} and \eqref{eq:observed-certificate-bound-app} were used. Thus
\begin{equation}\label{eq:noisy-normal-solved}
 \|\calP_{T^\perp}\tens E\|_*
 \le4C_{\rm W}\sqrt{11p^{-1}s w_{\max}\log N}\,\eps
 +\frac{1}{4N}\|\calP_T\tens E\|_F.
\end{equation}

Let $\calK=p^{-1}\calA_{\Omega}+\calA^\perp$. The tangent-space concentration event gives
\begin{align*}
 \|\calP_T\tens E\|_F
 \le\sqrt2\|\calK^{1/2}\calP_T\tens E\|_F
 \le\sqrt{2p^{-1}}\left(
 \|\calA_{\Omega}\tens E\|_F
 +\|\calP_{T^\perp}\tens E\|_*\right).
\end{align*}
Indeed, $\calK\preceq p^{-1}\calI$ and $\|\calK^{1/2}\tens E\|_F =p^{-1/2}\|\calA_{\Omega}\tens E\|_F$. Substitution of \eqref{eq:noisy-observed-lifted-2norm} and \eqref{eq:noisy-normal-solved}, followed by absorption of the term $(4N)^{-1}\sqrt{2/p}\|\calP_T\tens E\|_F$, yields
\begin{equation}\label{eq:noisy-tangent-2norm}
 \|\calP_T\tens E\|_F
 \le \widetilde C_3p^{-1}\sqrt{s w_{\max}\log N}\,\eps.
\end{equation}
Here the sampling condition, with $C_2$ enlarged if necessary, implies $p\ge N^{-1}$ and makes the absorption uniform for $N\ge2$. Combining \eqref{eq:noisy-normal-solved} and \eqref{eq:noisy-tangent-2norm}, and using $w_{\max}\le\min\{|A|,|B|\}$, gives
\[
 \|\tens E\|_2
 \le\|\calP_T\tens E\|_F+\|\calP_{T^\perp}\tens E\|_*
 \le C_3p^{-1}\sqrt{s\min\{|A|,|B|\}\log N}\,\eps.
\]
\end{proof}

\subsection{Proof of Theorem~\ref{thm:completion-localization}}
\label{subsec:proof-completion-localization}
\begin{proof}
Let $\tens M_\star=\tens U\tens\Sigma\tens V^*$ be a compact SVD, and set
$$
\tens E=\calH_{A,B}(\widehat{\vect x}-\vect x),
\qquad
\varepsilon_H=\frac{\norm{\tens E}_2}{\sigma_s(\tens M_\star)}.
$$
 By \eqref{eq:pencil-factor-new} and \eqref{eq:pencil-frame-new},
\begin{equation}\label{eq:signal-gap-lower}
 \sigma_s(\tens M_\star)
 \ge\frac{\sqrt{|A||B|}}{\mu_1}\min_j|\widetilde c_j|,
 \qquad
 \varepsilon_H\le
 \frac{\mu_1\|\tens E\|_2}
 {\sqrt{|A||B|}\min_j|\widetilde c_j|}.
\end{equation}
Set $\calQ=\calR^2$ and $\widehat{\calQ}=\widehat{\calR}^2$. Since $\calR(\vect t_j)=0$, we have $\nabla\calQ(\vect t_j)=0$. The true locations lie in the interior of $\mathcal F$, so for every $\vect u\in\R^2$, \eqref{eq:music-resolution-bound} and Taylor's formula along $\vect t_j+\theta\vect u$ give $\vect u^T\nabla^2\calQ(\vect t_j)\vect u\ge2c_{\rm M}^2\|\vect u\|_2^2$. Choose $r>0$ sufficiently small that the sets $\calU_j=\mathcal F\cap\{\vect t:\|\vect t-\vect t_j\|_2<r\}$ are disjoint, $\calQ$ is strongly convex on each convex set $\overline{\calU_j}$, and $d_{\rm per}(\vect t,\vect t_j)=\|\vect t-\vect t_j\|_2<\Delta/4$ for $\vect t\in\calU_j$.

By compactness and \eqref{eq:music-resolution-bound}, $\gamma:=\min_{\mathcal F\setminus\bigcup_j\calU_j}\calQ>0$. Since the squared residual depends continuously in $C^2(\mathcal F)$ on the two signal-space projectors, there is $\eta_{\rm M}>0$ such that projector perturbations of norm at most $\eta_{\rm M}$ preserve strong convexity on each $\overline{\calU_j}$ and give $\|\widehat{\calQ}-\calQ\|_\infty<\gamma/3$. Then $\widehat{\calQ}(\vect t_j)<\gamma/3$, whereas $\widehat{\calQ}\ge2\gamma/3$ outside $\bigcup_j\calU_j$. Thus the unique minimizer on each $\overline{\calU_j}$ lies in $\calU_j$, is a local minimum relative to $\mathcal F$, and has value below every local minimum outside these neighborhoods. Strong convexity excludes other local minima in each $\calU_j$, so these are precisely the $s$ smallest local minima. The same statements hold for $\widehat{\calR}$ because squaring is strictly increasing on $[0,\infty)$.

After decreasing the admissible constant $c_{\rm M}$, we may assume $c_{\rm M}\Delta\le\min\{1,\eta_{\rm M}\}$. Condition~\eqref{eq:stablecondition} then gives $\varepsilon_H\le c_{\rm M}\Delta/C_4\le1/C_4$. Wedin's theorem \cite{Wedin1972,StewartSun1990} gives a numerical constant $c_4>0$ such that, for $C_4$ sufficiently large, we have
\[
 \max\left\{
 \|\widehat{\tens U}\widehat{\tens U}^*-\tens U\tens U^*\|_2,
 \|\widehat{\tens V}\widehat{\tens V}^*-\tens V\tens V^*\|_2
 \right\}\le c_4\varepsilon_H.
\]
Increasing $C_4$ if necessary makes this error at most $c_{\rm M}\Delta\le\eta_{\rm M}$. Hence the $s$ minima selected in Algorithm~\ref{alg:general-domain-music} can be labeled so that $\widehat{\vect t}_j\in\calU_j$ and $\widehat{\calR}(\widehat{\vect t}_j)\le\widehat{\calR}(\vect t_j)$.

Since the steering vectors have norms $\sqrt{|A|}$ and $\sqrt{|B|}$, the reverse triangle inequality for the stacked residuals yields, after increasing $c_4$ if necessary,
\begin{equation}\label{eq:music-residual-perturbation}
 \sup_{\vect t\in\mathcal F}
 |\widehat{\calR}(\vect t)-\calR(\vect t)|
 \le c_4\varepsilon_H=:\delta.
\end{equation}
Since $\calR(\vect t_j)=0$, we have $\widehat{\calR}(\widehat{\vect t}_j)\le\delta$ and therefore $\calR(\widehat{\vect t}_j)\le2\delta$. Moreover, $\calU_j\subset\{\vect t:d_{\rm per}(\vect t,\vect t_j)<\Delta/4\}$, so the linear branch of \eqref{eq:music-resolution-bound} gives
\[
 d_{\rm per}(\widehat{\vect t}_j,\vect t_j)
 \le\frac{2c_4}{c_{\rm M}}\varepsilon_H.
\]
Substitution of \eqref{eq:signal-gap-lower} proves \eqref{eq:completion-location-final}.
\end{proof}

\subsection{Proofs of the technical lemmas}
\label{subsec:proofs-technical-lemmas}

\subsubsection{Proof of Lemma~\ref{lem:leverage-frame}}
\label{subsec:proof-leverage-frame}
\begin{proof}
Since $\range(\tens U)=\range(\tens V_A)$, $\tens U\tens U^*=\tens V_A(\tens V_A^*\tens V_A)^{-1}\tens V_A^*$. Every row of $\tens V_A$ has squared norm $s$, and therefore $\|\tens U^*\vect e_\alpha\|_2^2\le s/\lambda_{\min}(\tens V_A^*\tens V_A)\le\mu_1s/|A|$. Since the right singular space is $\range(\overline{\tens V_B})$, the same argument gives the second inequality in \eqref{eq:leverage-frame-new}.
\end{proof}

\subsubsection{Proof of Lemma~\ref{lem:rectangular-fiber-bound}}
\label{subsec:proof-rectangular-fiber-bound}
\begin{proof}
Write $a_i:=|I_i|$, $b_i:=|J_i|$, and $m_i:=\min\{a_i,b_i\}$. By translation invariance, we may assume $I_i=\{0,\ldots,a_i-1\}$ and $J_i=\{0,\ldots,b_i-1\}$. Then
\[
w_i(t):=w_{I_i,J_i}(t)=\min\{t+1,a_i,b_i,a_i+b_i-1-t\},\qquad 0\le t\le a_i+b_i-2.
\]
For each nonempty monotonicity part $\sigma$ of $w_i$ and each integer $r\ge0$, define
\[
\mathcal L_{i,r}^{\sigma}
 :=\{t\text{ in part }\sigma:2^{r-1}<w_i(t)\le2^r\}
\]
and, for $D\subset I_i+J_i$, define
\[
\operatorname{Row}_i(D)
 :=\{u\in I_i:\text{ there are }t\in D\text{ and }v\in J_i
 \text{ with }u+v=t\}.
\]
Then, for every nonempty $\mathcal L_{i,r}^{\sigma}$, we have
\begin{equation}\label{eq:one-dimensional-row-incidence}
|\operatorname{Row}_i(\mathcal L_{i,r}^{\sigma})|\le 2c_{\rm rec}2^r.
\end{equation}
In fact, if $\mathcal L_{i,r}^{\sigma}$ lies in the increasing part, then $\operatorname{Row}_i(\mathcal L_{i,r}^{\sigma})\subset\{0,\ldots,2^r-1\}$; if it lies in the decreasing part, then $\operatorname{Row}_i(\mathcal L_{i,r}^{\sigma})$ is contained in a terminal interval of length at most $2^{r+1}$; if it lies in the plateau, then $2^{r-1}<m_i\le2^r$ and $|\operatorname{Row}_i(\mathcal L_{i,r}^{\sigma})|\le a_i\le c_{\rm rec}m_i\le c_{\rm rec}2^r$.

Now fix $\mathcal B:=\mathcal L_{1,r_1}^{\sigma_1}\times\mathcal L_{2,r_2}^{\sigma_2}$. For every $k\in\mathcal B$, $w(k)=w_1(k_1)w_2(k_2)>2^{r_1+r_2-2}$, and all corresponding matrix entries lie in rows indexed by $\operatorname{Row}_1(\mathcal L_{1,r_1}^{\sigma_1})\times\operatorname{Row}_2(\mathcal L_{2,r_2}^{\sigma_2})$, whose cardinality is at most $4c_{\rm rec}^2 2^{r_1+r_2}$ by \eqref{eq:one-dimensional-row-incidence}. Hence
\begin{align*}
\sum_{k\in\mathcal B}\frac{1}{w(k)^2 }\Big|\sum_{\alpha+\beta=k}M_{\alpha,\beta}\Big|^2
&\le \sum_{k\in\mathcal B}\frac1{w(k)}\sum_{\alpha+\beta=k}|M_{\alpha,\beta}|^2
\le 2^{2-r_1-r_2}\sum_{\substack{k\in\mathcal B\\ \alpha+\beta=k}}|M_{\alpha,\beta}|^2 \\
&\le 2^{2-r_1-r_2}\,4c_{\rm rec}^2 2^{r_1+r_2}\norm{\tens M}_{\infty,2}^2
=16c_{\rm rec}^2\norm{\tens M}_{\infty,2}^2,
\end{align*}
where the first inequality follows from Cauchy--Schwarz on each fiber, the second follows from $w(k)>2^{r_1+r_2-2}$, and the third follows from the disjointness of the fibers and the preceding row count. There are at most $9$ choices of $(\sigma_1,\sigma_2)$ and at most $4(1+\log|A_0+B_0|)^2$ nonempty dyadic level pairs. Summing over all classes gives
\[
\norm{\tens M}_{\calG,2}^2
\le 576c_{\rm rec}^2(1+\log|A_0+B_0|)^2\norm{\tens M}_{\infty,2}^2,
\]
which proves \eqref{eq:rect-fiber-energy}.
\end{proof}

\subsubsection{Proof of Lemma~\ref{lem:kernel-bounds}}
\label{subsec:proof-kernel-bounds}

\begin{proof}
Recall $\tens G_k=w(k)^{-1/2}\sum_{\alpha+\beta=k}\vect e_\alpha \vect e_\beta^{*}$. Since the summands have disjoint supports, we have
\begin{align*}
\norm{\tens U\tens U^*\tens G_k}_F^2
&=\frac1{w(k)}\sum_{\alpha+\beta=k}
  \norm{\tens U^*\vect e_\alpha}_2^2
 \le \frac{\mu_1s}{|A|},\\
\norm{\tens G_k\tens V\tens V^*}_F^2
&=\frac1{w(k)}\sum_{\alpha+\beta=k}
  \norm{\tens V^*\vect e_\beta}_2^2
 \le \frac{\mu_1s}{|B|}.
\end{align*}
These estimates follow from \Cref{lem:leverage-frame}. By \eqref{eq:tangent-projection-formula-app}, we have
\[
\norm{\calP_T\tens G_k}_F^2
\le \norm{\tens U\tens U^*\tens G_k}_F^2
+\norm{\tens G_k\tens V\tens V^*}_F^2
\le \mu_1s\left(\frac1{|A|}+\frac1{|B|}\right)
\le \frac{2\mu_1s c_{A,B}}{N},
\]
which proves the first inequality in \eqref{eq:tangent-atom-bound-app}. For the kernel estimate, expanding \eqref{eq:tangent-projection-formula-app} gives
\begin{align*}
|\inner{\tens G_b}{\calP_T\tens G_a}|
&\le |\inner{\tens G_b}{\tens U\tens U^*\tens G_a}|
+|\inner{\tens G_b}{\tens G_a\tens V\tens V^*}|
+|\inner{\tens G_b}{\tens U\tens U^*\tens G_a\tens V\tens V^*}| \\
&\le \sqrt{\frac{w(b)}{w(a)}}
\left[
\max_{\alpha,\alpha'}|(\tens U\tens U^*)_{\alpha,\alpha'}|
+\max_{\beta,\beta'}|(\tens V\tens V^*)_{\beta,\beta'}|
+\frac{\mu_1s}{\sqrt{|A||B|}}
\right]\\
&\le \sqrt{\frac{w(b)}{w(a)}}\mu_1s
\left(\frac1{|A|}+\frac1{|B|}+\frac1{\sqrt{|A||B|}}\right)\\
&\le 3\sqrt{\frac{w(b)}{w(a)}}
\frac{\mu_1s c_{A,B}}{N}.
\end{align*}
For the first two terms, each nonzero entry of $\tens G_b$ meets at most one nonzero entry in the corresponding row or column of $\tens G_a$, and
\[
|(\tens U\tens U^*)_{\alpha,\alpha'}|\le\frac{\mu_1s}{|A|},
\qquad
|(\tens V\tens V^*)_{\beta,\beta'}|\le\frac{\mu_1s}{|B|}.
\]
For the last term,
\[
|\vect e_\alpha^*\tens U\tens U^*\tens G_a\tens V\tens V^*\vect e_\beta|
\le\frac{\mu_1s}{\sqrt{|A||B|w(a)}}.
\]
Summing over the $w(b)$ nonzero entries of $\tens G_b$ and using
\[
\max\bigl\{|A|^{-1},|B|^{-1},(|A||B|)^{-1/2}\bigr\}
\le \frac{c_{A,B}}{N}
\]
proves the second inequality.
\end{proof}

\subsubsection{Proof of Lemma~\ref{lem:initial-structural-norms}}
\label{subsec:proof-initial-structural-norms}
\begin{proof}
For every matrix entry, we have
\[
        |(\tens{Q}_\star)_{\alpha,\beta}|
        \le\norm{\tens{U}^*\vect{e}_\alpha}_2
             \norm{\tens{V}^*\vect{e}_\beta}_2
        \le\frac{\mu_1s}{\sqrt{|A||B|}}
        \le\frac{\mu_1c_{A,B}s}{N}.
\]
Since $\frac{|\inner{\tens{G}_k}{\tens{Q}_\star}|}{\sqrt{w(k)}} =\left|\frac1{w(k)} \sum_{\alpha+\beta=k}(\tens{Q}_\star)_{\alpha,\beta}\right|$, this proves the first inequality in \eqref{eq:Q-app}.  In addition,
\[
\max_\alpha\norm{\vect e_\alpha^*\tens Q_\star}_2^2
=\max_\alpha\norm{\vect e_\alpha^*\tens U}_2^2
\le\frac{\mu_1s}{|A|}.
\]
The analogous column bound is $\mu_1s/|B|$. Therefore $\norm{\tens{Q}_\star}_{\infty,2}^2 \le\mu_1c_{A,B}s/N$. By the definition of $\Gamma_{A,B}$, this gives the second inequality in \eqref{eq:Q-app}. For the last inequality in \eqref{eq:Q-app}, \eqref{eq:tangent-projection-formula-app} gives, for every row index $\alpha$,
\begin{align*}
\norm{\vect e_\alpha^*\tens U\tens U^*\tens P_a}_2
\le\norm{\tens U^*\vect e_\alpha}_2,\quad
\norm{\vect e_\alpha^*\tens P_a\tens V\tens V^*}_2
\le\max_{\beta\in B}\norm{\tens V^*\vect e_\beta}_2,\quad
\norm{\vect e_\alpha^*\tens U\tens U^*\tens P_a\tens V\tens V^*}_2
\le\norm{\tens U^*\vect e_\alpha}_2.
\end{align*}
The second inequality uses the fact that a row of $\tens P_a$ is either zero or a standard basis row. Hence
\[
        \max_\alpha
        \norm{\vect{e}_\alpha^*\calP_T\tens{P}_a}_2^2
        \le 5\mu_1s\left(\frac1{|A|}+\frac1{|B|}\right)
        \le\frac{10\mu_1c_{A,B}s}{N}.
\]
The corresponding column bound follows by symmetry. Thus $\norm{\calP_T\tens{P}_a}_{\infty,2}^2 \le 10\mu_1c_{A,B}s/N$, and we complete the proof by the definition of $\Gamma_{A,B}$.
\end{proof}

\subsubsection{Proof of Lemma~\ref{lem:completion-certificate-event}}
\label{subsec:proof-completion-certificate-event}
\begin{proof}
We follow the golfing-scheme proof of \cite{ChenChi2014}. We first record three estimates for one independent Bernoulli batch. Fix $0<\delta_b<1$, let $\Omega'$ be sampled with probability $p_{\rm b}$, and write $\zeta_k=\one_{\{k\in\Omega'\}}$. Define
\[
        \mathcal X_k:=\left(\zeta_k/p_{\rm b}-1\right)
        (\calP_T\tens G_k)\otimes(\calP_T\tens G_k).
\]
Then $\calP_T(p_{\rm b}^{-1}\calA_{\Omega'}-\calA)\calP_T =\sum_{k\in S}\mathcal X_k$. Since $|A|,|B|\le N$, the ambient dimensions in the matrix Bernstein bounds below contribute only numerical multiples of $\log(N/\delta_b)$. Lemma~\ref{lem:kernel-bounds} gives
\[
 \|\mathcal X_k\|\le\frac{2\mu_1c_{A,B}s}{p_{\rm b}N},
 \qquad
 \left\|\sum_{k\in S}\mathbb E\mathcal X_k^2\right\|
 \le\frac{2\mu_1c_{A,B}s}{p_{\rm b}N}.
\]
Matrix Bernstein \cite[Theorem~1.4]{Tropp2012} therefore gives, with probability at least $1-\delta_b$,
\begin{equation}\label{eq:batch-tangent-concentration}
 \|\calP_T(p_{\rm b}^{-1}\calA_{\Omega'}-\calA)\calP_T\|\le\frac12,
\end{equation}
provided $p_{\rm b}\ge c_1\mu_1c_{A,B}s\log(N/\delta_b)/N$ for a numerical constant $c_1>0$. Fix $\tens M\in T$ independent of $\Omega'$, and set
\[
        \tens H:=\calP_T(\calA-p_{\rm b}^{-1}\calA_{\Omega'})\tens M
        =\sum_{a\in S}\left(1-\zeta_a/p_{\rm b}\right)
        \inner{\tens G_a}{\tens M}\,\calP_T\tens G_a .
\]
Define the mixed seminorm
\[
        \Phi(\tens M):=\norm{\tens M}_{\calG,2}
        +\sqrt{\frac{N}{\mu_1c_{A,B}\Gamma_{A,B}s}}
        \norm{\tens M}_{\calG,\infty}.
\]
For the $\calG,2$-part, Lemma~\ref{lem:initial-structural-norms} gives $\norm{\calP_T\tens G_a}_{\calG,2}^2 =\frac{1}{w(a)}\norm{\calP_T\tens P_a}_{\calG,2}^2 \le \frac{10\mu_1c_{A,B}\Gamma_{A,B}s}{Nw(a)}$. Define the $|S|\times1$ random vector
\[
\vect X_a:=\left(1-\frac{\zeta_a}{p_{\rm b}}\right)
\inner{\tens G_a}{\tens M}
\left(\frac{\inner{\tens G_k}{\calP_T\tens G_a}}{\sqrt{w(k)}}\right)_{k\in S}.
\]
Then $\sum_{a\in S}\vect X_a =(\inner{\tens G_k}{\tens H}/\sqrt{w(k)})_{k\in S}$, and hence $\|\sum_{a\in S}\vect X_a\|_2=\|\tens H\|_{\calG,2}$. Thus
\[
        \|\vect X_a\|_2
        \le \frac{\sqrt{10}}{p_{\rm b}}
        \sqrt{\frac{\mu_1c_{A,B}\Gamma_{A,B}s}{N}}
        \norm{\tens M}_{\calG,\infty},
        \qquad
        \sum_{a\in S}\mathbb E\|\vect X_a\|_2^2
        \le \frac{10\mu_1c_{A,B}\Gamma_{A,B}s}{p_{\rm b}N}
        \norm{\tens M}_{\calG,2}^2.
\]
Matrix Bernstein \cite[Theorem~1.6]{Tropp2012} gives, for a numerical constant $c_2>0$, with probability at least $1-\delta_b$,
\[
        \norm{\tens H}_{\calG,2}
        \le c_2\sqrt{\frac{\mu_1c_{A,B}\Gamma_{A,B}s\log(N/\delta_b)}{p_{\rm b}N}}
        \norm{\tens M}_{\calG,2}
        +c_2\frac{\log(N/\delta_b)}{p_{\rm b}}
        \sqrt{\frac{\mu_1c_{A,B}\Gamma_{A,B}s}{N}}
        \norm{\tens M}_{\calG,\infty}.
\]
For the $\calG,\infty$ part, Lemma~\ref{lem:kernel-bounds} gives, for all $a,b\in S$, $\frac{\abs{\inner{\tens G_b}{\calP_T\tens G_a}}}{\sqrt{w(b)}} \le\frac{3\mu_1c_{A,B}s}{N\sqrt{w(a)}}$. Thus, for each fixed $b$, the scalar Bernstein summands in $\inner{\tens G_b}{\tens H}/\sqrt{w(b)}$ obey
\begin{align*}
   \abs{\left(1-\frac{\zeta_a}{p_{\rm b}}\right)
        \inner{\tens G_a}{\tens M}
        \frac{\inner{\tens G_b}{\calP_T\tens G_a}}{\sqrt{w(b)}}}
        &\le \frac{3\mu_1c_{A,B}s}{p_{\rm b}N}
        \norm{\tens M}_{\calG,\infty},\\
        \sum_{a\in S}\mathbb E\abs{\left(1-\frac{\zeta_a}{p_{\rm b}}\right)
        \inner{\tens G_a}{\tens M}
        \frac{\inner{\tens G_b}{\calP_T\tens G_a}}{\sqrt{w(b)}}}^2
        &\le \frac{9\mu_1^2c_{A,B}^2s^2}{p_{\rm b}N^2}
        \norm{\tens M}_{\calG,2}^2. 
\end{align*}
Applying scalar Bernstein \cite[Theorem~7.30]{FoucartRauhut2013} to the real and imaginary parts, followed by a union bound over $b\in S$, gives with probability at least $1-\delta_b$,
\[
        \norm{\tens H}_{\calG,\infty}
        \le c_2\frac{\mu_1c_{A,B}s}{N}
        \sqrt{\frac{\log(N/\delta_{b})}{p_{\rm b}}}\norm{\tens M}_{\calG,2}
        +c_2\frac{\mu_1c_{A,B}s\log(N/\delta_{b})}{p_{\rm b}N}
        \norm{\tens M}_{\calG,\infty}.
\]
Thus, whenever $p_{\rm b}\ge (16c_2^2+4c_2)\mu_1c_{A,B}\Gamma_{A,B}s\log(N/\delta_{b})/N$, we have
\begin{equation}\label{eq:batch-contraction}
        \Phi(\tens H)\le \frac12\Phi(\tens M).
\end{equation}
For the normal-space leakage, set
\[
 \tens Z_a=\left(\zeta_a / p_{\rm b}-1\right)
 \inner{\tens G_a}{\tens M}\calP_{T^\perp}\tens G_a.
\]
Then
\begin{align*}
 \|\tens Z_a\|_2\le p_{\rm b}^{-1}\|\tens M\|_{\calG,\infty},\quad
 \max\left\{
 \left\|\sum_{a\in S}\mathbb E\tens Z_a\tens Z_a^*\right\|,
 \left\|\sum_{a\in S}\mathbb E\tens Z_a^*\tens Z_a\right\|
 \right\}\le p_{\rm b}^{-1}\|\tens M\|_{\calG,2}^2.
\end{align*}
A second application of matrix Bernstein \cite[Theorem~1.6]{Tropp2012} gives
\begin{equation}\label{eq:batch-leakage}
 \|\calP_{T^\perp}(p_{\rm b}^{-1}\calA_{\Omega'}-\calA)\tens M\|_2
 \le c_3\left[
 \sqrt{\frac{\log(N/\delta_b)}{p_{\rm b}}}\|\tens M\|_{\calG,2}
 +\frac{\log(N/\delta_b)}{p_{\rm b}}\|\tens M\|_{\calG,\infty}\right]
\end{equation}
with probability at least $1-\delta_b$, for a numerical constant $c_3>0$. We now apply the tangent estimate, the two mixed-norm estimates, and the leakage estimate in a golfing construction. Let $J=\lceil\log_2(8N\sqrt{s})\rceil$, $p_{\rm b}=1-(1-p)^{1/J}$, and draw independent Bernoulli-$p_{\rm b}$ sets $\Omega_1,\ldots,\Omega_J$.  Their union has the same law as $\Omega$, $p_{\rm b}\ge p/J$, and, since $s\le\min\{|A|,|B|\}\le N$ and $\delta<1/2$, $J\le6\log(N/\delta)$. Set $\delta_{b}=\delta/[4(J+1)]$. Since $\Gamma_{A,B}\ge1$, the assumed lower bound on $p$ implies all required lower bounds on $p_{\rm b}$, provided $C_0$ is sufficiently large in terms of $c_1,c_2,c_3$. Applying each estimate conditionally on the preceding batches, the batch estimates hold for all $j$ with total failure probability at most $(4J+1)\delta_b\le\delta$.  Define $\tens Q_0=\tens Q_\star$ and, for $j\ge1$,
\[
 \tens Q_j=\calP_T(\calA-p_{\rm b}^{-1}\calA_{\Omega_j})\tens Q_{j-1},
        \qquad
        \tens W:=\sum_{j=1}^J(p_{\rm b}^{-1}\calA_{\Omega_j}+\calA^\perp)
        \tens Q_{j-1}.
\]
Since $\tens Q_{j-1}$ is independent of $\Omega_j$, \eqref{eq:batch-contraction} gives $\Phi(\tens Q_j)\le2^{-j}\Phi(\tens Q_0)$.  Moreover, $\norm{\tens Q_j}_F\le2^{-j}\sqrt{s}$ by \eqref{eq:batch-tangent-concentration}, and Lemma~\ref{lem:initial-structural-norms} and $\Gamma_{A,B}\ge1$ give $\Phi(\tens Q_0)\le2\sqrt{\mu_1c_{A,B}\Gamma_{A,B}s/N}$. The tangent component telescopes:
\begin{align*}
\calP_T\tens W
=\sum_{j=1}^J\calP_T(p_{\rm b}^{-1}\calA_{\Omega_j}+\calA^\perp)\tens Q_{j-1}
=\sum_{j=1}^J(\tens Q_{j-1}-\tens Q_j)
=\tens Q_\star-\tens Q_J.
\end{align*}
Thus $\calP_T(\tens W-\tens Q_\star)=-\tens Q_J$, and the choice of $J$ gives $\norm{\calP_T(\tens W-\tens Q_\star)}_F\le1/(8N)$.  For the normal component, since $\tens Q_{j-1}\in T$,
\[
 \calP_{T^\perp}(p_{\rm b}^{-1}\calA_{\Omega_j}+\calA^\perp)\tens Q_{j-1}
 =\calP_{T^\perp}(p_{\rm b}^{-1}\calA_{\Omega_j}-\calA)\tens Q_{j-1}.
\]
Since $\norm{\tens Q}_{\calG,2}\le\Phi(\tens Q)$ and
\[
\norm{\tens Q}_{\calG,\infty}
\le\sqrt{\frac{\mu_1c_{A,B}\Gamma_{A,B}s}{N}}\,\Phi(\tens Q),
\]
summing \eqref{eq:batch-leakage} and using the geometric decay of $\Phi(\tens Q_{j-1})$ gives
\[
\norm{\calP_{T^\perp}\tens W}_2
\le4c_3\left(
\sqrt{\frac{\mu_1c_{A,B}\Gamma_{A,B}s\log(N/\delta_b)}{p_{\rm b}N}}
+\frac{\mu_1c_{A,B}\Gamma_{A,B}s\log(N/\delta_b)}{p_{\rm b}N}
\right)\le\frac12
\]
for $C_0$ sufficiently large. Finally, $\tens W\in\range(\calA_\Omega)+\range(\calA^\perp)$. Since $\calA_\Omega\calA^\perp=0$ and $\Omega_j\subset\Omega$, we have
\[
        \norm{\calA_\Omega\tens W}_F
        \le\sum_{j=1}^J
        \norm{p_{\rm b}^{-1}\calA_{\Omega_j}\tens Q_{j-1}}_F .
\]
For $\tens Q_{j-1}\in T$, \eqref{eq:batch-tangent-concentration} gives
\[
\norm{p_{\rm b}^{-1}\calA_{\Omega_j}\tens Q_{j-1}}_F^2
=p_{\rm b}^{-1}\inner{\tens Q_{j-1}}
 {p_{\rm b}^{-1}\calA_{\Omega_j}\tens Q_{j-1}}
\le \frac32p_{\rm b}^{-1}\norm{\tens Q_{j-1}}_F^2.
\]
Therefore
\begin{align*}
\norm{\calA_\Omega\tens W}_F
\le \sqrt{\frac32}\,p_{\rm b}^{-1/2}
        \sum_{j=1}^J\norm{\tens Q_{j-1}}_F
\le 2\sqrt{\frac32}\sqrt{\frac{sJ}{p}}
 \le 6\sqrt{\frac{s\log(N/\delta)}{p}}\le C_{\rm W}\sqrt{\frac{s\log(N/\delta)}{p}}.
\end{align*}
Applying \eqref{eq:batch-tangent-concentration} directly to the full Bernoulli-$p$ set gives \eqref{eq:full-sample-tangent-event-app}.  This completes the proof for $C_0$ sufficiently large.
\end{proof}

\bibliographystyle{plain}
\bibliography{references}

\end{document}